\documentclass[11pt]{article}
\usepackage{amsmath, latexsym, amsfonts, amssymb, amsthm, amscd}
\usepackage{amsmath}
\usepackage{mathtools} 
\usepackage[left=2.5cm, right=2.5cm, top=2.5cm, bottom=2.5cm]{geometry}

\usepackage{upquote} 
\usepackage{color}
\usepackage{setspace} 
\usepackage[labelfont=bf]{caption} 
\usepackage{stmaryrd} 
\usepackage{multicol} 
\usepackage[dvipsnames]{xcolor}
\usepackage{lipsum}
\usepackage{kpfonts} 
\usepackage{dsfont} 
\usepackage{url}
\usepackage[english]{babel} 
\usepackage[utf8]{inputenc}
\usepackage[T1]{fontenc}
\usepackage{enumerate}
\usepackage{graphicx}
\usepackage{color,soul}
\usepackage{comment}
\usepackage[toc,page]{appendix}
\usepackage[mathlines]{lineno}
\usepackage{etoolbox} 
\newcommand*\linenomathpatch[1]{%
	\cspreto{#1}{\linenomath}%
	\cspreto{#1*}{\linenomath}%
	\csappto{end#1}{\endlinenomath}%
	\csappto{end#1*}{\endlinenomath}%
}
\newcommand*\linenomathpatchAMS[1]{%
	\cspreto{#1}{\linenomathAMS}%
	\cspreto{#1*}{\linenomathAMS}%
	\csappto{end#1}{\endlinenomath}%
	\csappto{end#1*}{\endlinenomath}%
}
\expandafter\ifx\linenomath\linenomathWithnumbers
\let\linenomathAMS\linenomathWithnumbers
\patchcmd\linenomathAMS{\advance\postdisplaypenalty\linenopenalty}{}{}{}
\else
\let\linenomathAMS\linenomathNonumbers
\fi

\linenomathpatch{equation}
\linenomathpatchAMS{gather}
\linenomathpatchAMS{multline}
\linenomathpatchAMS{align}
\linenomathpatchAMS{alignat}
\linenomathpatchAMS{flalign}

\usepackage{bm} 

\usepackage{enumitem} 

\usepackage{xcolor}
\definecolor{Maroon}{HTML}{ad2231}
\definecolor{webgreen}{HTML}{008000}
\usepackage{dsfont}
\usepackage{hyperref}
\usepackage{lipsum}
\hypersetup{colorlinks, breaklinks, urlcolor=black, linkcolor=black, citecolor=webgreen} 
\allowdisplaybreaks
\usepackage{marginnote}

\usepackage{amsthm}
\newtheorem{theorem}{Theorem}

\newtheorem{proposition}[theorem]{Proposition}
\newtheorem{lemma}[theorem]{Lemma}
\newtheorem{remark}[theorem]{Remark}

\theoremstyle{definition}

\usepackage{todonotes}

\usepackage{chngcntr}
\usepackage{apptools}
\newcommand{\email}[1]{\gdef\@email{\url{#1}}}

\newcommand{\N}{\mathbb{N}}

\newcommand{\R}{\mathbb{R}}
\newcommand{\W}{\mathbb{W}}
\newcommand{\Z}{\mathbb{Z}}

\newcommand{\PP}{\mathbb{P}}
\newcommand{\E}{\mathbb{E}}

\DeclareMathAlphabet{\mathpzc}{OT1}{pzc}{m}{it}
\newcommand{\ee}{\textnormal{e}}
\newcommand{\dd}{\textnormal{d}}
\newcommand{\jj}{\textnormal{j}}

\newcommand{\1}{\mathbf{1}}

\newcommand{\wt}{\widetilde}
\newcommand{\ovl}{\overline}

\newcommand{\eqsp}{\;\;\;\;}
\makeatletter
\def\widebreve{\mathpalette\wide@breve}
\def\wide@breve#1#2{\sbox\z@{$#1#2$}%
	\mathop{\vbox{\m@th\ialign{##\crcr
				\kern0.08em\brevefill#1{0.8\wd\z@}\crcr\noalign{\nointerlineskip}%
				$\hss#1#2\hss$\crcr}}}\limits}
\def\brevefill#1#2{$\m@th\sbox\tw@{$#1($}%
	\hss\resizebox{#2}{\wd\tw@}{\rotatebox[origin=c]{90}{\upshape(}}\hss$}
\makeatletter
\allowdisplaybreaks
\begin{document}
		\title{Poissonian potential measures  for refracted-reflected L\'evy processes}
	\author{
		Noah Beelders\footnote{Department of Mathematical Sciences,
			University of Liverpool, \texttt{N.Beelders@liverpool.ac.uk}}, \,   \, Lewis Ramsden\footnote{School for Businesses and Society, Univeristy of York, \texttt{lewis.ramsden@york.ac.uk}} \;   \;  \&
		\, Apostolos D. Papaioannou\footnote{Department of Mathematical Sciences,
			University of Liverpool, \texttt{papaion@liverpool.ac.uk}}   \;
	}
	
	\maketitle
	
	\tableofcontents
	
		\vspace{0.1in}
	
	\begin{abstract} 
		In this paper we study  the potential measures and the Laplace transforms of the occupation times of a refracted-reflected spectrally negative  L\'evy process when the process is observed at the arrival epochs of two independent Poisson processes. In this case, the rates of observing the underlying process differ in time which deviates from the classical theory of Poissonian observations.  Explicit expressions for the so-called Poissonian potential measures and the Poissonian  occupation times are derived  in terms of (known) scale  functions.  Other fluctuation identities  are also derived. 
	\end{abstract}
	
	
	
	\section{Introduction}
	The refracted-reflected L\'evy process  can be described  as a combination  of a L\'evy process reflected at a lower boundary  and  a refracted L\'evy process. The former can be expressed as the difference between the underlying L\'evy process and its running infimum (see \cite{B1996, K2014, P2004}), and the latter as the underlying process below a fixed level and a drift changed process above this level (see \cite{KL2010}). Refracted-reflected L\'evy processes have several applications in applied probability, queuing theory and insurance mathematics. For example, they are used for the bail-out model in the presence of dividends, where the surplus is a L\'evy process for which dividends are paid at a fixed level and capital is injected to prevent the process going below 0 (see for e.g. \cite{ JMPY2019, PY2017, PY2018}). 
	
	Recently, motivated by insurance mathematics  and  reliability theory problems, developments in fluctuation theory have been made when the process is monitored at Poissonian time epochs (that are independent from the process). The applicability  of these Poissonian observations  can be found, for example,  in the insurance context where  the ruin time  or the time of dividend payments do not occur at the exact times the surplus process crosses a certain level (for instance 0), but rather at some time afterwards which more closely reflects the behaviour of financial institutions, see \cite{ACT2011, ACT2013,   LRZ2011}.  
	In the L\'evy context, Poissonian observations were first introduced by \cite{AIZ2016}, in which explicit expressions for the exit problems were derived, whilst  in \cite{LLWX2018} explicit expressions for the Poissonian potential measures were established. Poissonian observations are also relevant  in queueing contexts, see for example \cite{BBR2009}. 
	
	Potential measures are known to play a fundamental role in fluctuation theory of L\'evy processes (both under continuous and Poissonian observation frameworks), see for example \cite{KL2010,P2005} to mention a few. Surprisingly, analytic expressions for the   potential measures and the Laplace transforms of the occupation times  for the  refracted-reflected L\'evy process under Poissonian observations do not exist in the literature. 
	
	In this paper, we derive explicit expressions in terms of the so-called scale functions (see Section \ref{prel}) for the Poissonian potential measures and the occupation times for  a refracted-reflected   spectrally negative L\'evy process when the underlying process  is observed at the arrival epochs of two independent Poisson processes under the assumption that the rates of these observations differ. It will be seen that these derived expressions generalise those from the models in \cite{ACT2011, ACT2013, AIZ2016, LLWX2018}, and yield analogous forms to the Poissonian observation fluctuation identities in the literature. 
	With regards to the applicability of such a model,  the observation structure  yields a more accurate representation of reality in problems arising in insurance, queueing systems or finance: (i) In insurance, the proposed process can be used as a risk model  with dividends and capital injections for an insurance firm, in which the different Poissonian observations represent two (independent) observers that check the surplus solvency - a typical example in practice being the compliance officers (within the firm) and the regulators (outside the firm). This mechanism of multiple independent observers extends the current literature, see \cite{ACT2011, ACT2013, LLWX2018}.  (ii) In queuing systems,  the  proposed model is used  in systems where the workload is a L\'evy driven fluid queue (see \cite{KM2015b}, \cite{PVZ2014}) with state dependent services and non-negative constraints. Indeed, the Poisson clocks represent two different observers  looking for  inspection or demand  events  and the  restocking or service activation.  (iii) In finance, for a Lévy-driven asset, where  one clock may represent a liquidation agent who monitors the asset for distress signals and the other a regulatory authority conducting periodic compliance checks.  
	
	The paper is organised as follows. In Section \ref{prel}, we review the theory of scale functions and develop new simplifying formulae between known scale functions as well as some results that will be used in the subsequent sections and potential future works. In Section \ref{main},  we compute the two-sided  killed Poissonian potential measures, which we use in Section \ref{Sec:LimitsofMainTheorems} to provide expressions for the  one-sided killed Poissonian potential measures. Finally, in Section \ref{sec:occup}, we compute the Laplace transform of the Poissonian occupation times.

	\section{Preliminaries on scale functions}
	\label{prel}
	Let $X = \{X_{t}\}_{t \geq 0}$ be a spectrally negative L\'evy process (SNLP) defined on  $(\Omega, \mathcal{F}, \{\mathcal{F}_{t}\}_{t \geq 0}, \mathbb{P})$, where the filtration $\{\mathcal{F}_{t}\}_{t \geq 0}$ is assumed to satisfy the usual assumptions of right-continuity and completion. We shall denote $\PP_x$ to be the probability measure given that the process under consideration starts at $x$,  and  by  $\E_x$ the corresponding  expectation. We denote by $\PP$ and $\E$ the case when $x=0$. A L\'evy process with no positive jumps (the case of monotone paths is excluded) has a Laplace exponent  $\psi(\vartheta):[0,\infty)\rightarrow \mathbb R$ such that $	\psi(\vartheta) := \log \E[e^{\vartheta X_1}]$,  $\forall \vartheta\geqslant 0$, where $ \psi(\vartheta) = \mu\vartheta + \frac{\vartheta^2 \sigma^2}{2} + 
	\int_{(-\infty, 0)}^{} (e^{\vartheta x} -1- \vartheta x \mathbf{1}_{\{x >- 1\}}) \nu(\mathrm{d}x)$, 
	with  $\mu \in \R$, $\sigma \geq 0$ 
	and $\nu$, the L\'evy measure, is a $\sigma$-finite measure 
	concentrated on $(-\infty,0)$ satisfying $\int_{(-\infty,0)}(1 \wedge |x|^2) \nu(\mathrm{d}x)  <\infty$. 
	
	It is well-known that the fluctuation theory for $X$ relies heavily on the so-called $W$ and $Z$ \emph{scale functions} (see Chapter 8 in \cite{K2014}). For any $q\geq 0$, define   the  so-called scale functions  $W^{(q)}: \R \to [0, \infty)$ and $Z^{(q)}: \mathbb{R} \rightarrow[1, \infty)$ as 
	\begin{equation}
		\int_{0}^{\infty} e^{-\vartheta x}W^{(q)}(x)\mathrm{d}x = \frac{1}{\psi_q(\vartheta)}, \quad  \vartheta > \Phi_q, \quad \text{and}\quad  Z^{(q)}(x)=1+q \int_0^x W^{(q)}(y) \mathrm{d} y, 
		\label{eq:LTofscaleW}
	\end{equation}
	where $\Phi_q:= \sup \{ \vartheta \geq 0: \psi(\vartheta) = q\}$, $\psi_q(\vartheta) := \psi(\vartheta) -q$ and $W^{(q)}(x)=0$ for $x<0$. In the rest of the paper, we write $W$ or $\psi$ instead of $W^{(0)}$ or $\psi_0$ for convenience. We define also the  bivariate generalisation $Z^{(q)}: \mathbb{R} \times [0,\infty) \rightarrow[1, \infty)$  having the forms 
	\begin{equation} 	Z^{(q)}(x,\vartheta)=\ee^{\vartheta x} \Bigl(1-\psi_q(\vartheta)  \int_0^x \ee^{-\vartheta y}W^{(q)}(y) \mathrm{d} y \Bigr)=\psi_q(\vartheta)  \int_0^\infty \ee^{-\vartheta y}W^{(q)}(x+y) \mathrm{d} y, \quad x \geq 0, \label{eq:GeneralisedZFunc}
	\end{equation}
	\color{black} 
	where $Z^{(q)}(x,0) = Z^{(q)}(x)$ and $Z^{(q)}(x,\vartheta) = \ee^{\vartheta x}$ for $x \leq 0$.
	With regards to the limits of scale functions, it is well-known (see, for instance, Eqs.~(2.12) and (2.13) in \cite{LLWX2018}) that 
	\begin{equation}
		\lim_{a \rightarrow \infty} \frac{W^{(q)}(a+x)}{W^{(q)}(a)} = e^{\Phi_q x}, \; \quad \lim_{a \rightarrow \infty} \frac{Z^{(q)}(a)}{W^{(q)}(a)} = \frac{q}{\Phi_q}. \label{eq:LimitofRatioofScaleFuncs}
	\end{equation}
	In addition, the following useful identities for convolutions of the scale functions will be used throughout the paper. For any $p, q, x \geq 0$ and $p \neq q$, it holds that  (see \cite{LLWX2018} or \cite{LRZ2014})   
	\begin{align} 
		(p-q)\int_0^x W^{(p)}(x-y) W^{(q)}(y) \mathrm{d} y&=W^{(p)}(x)-W^{(q)}(x),  \label{eq:LoeffenLTIdentity1} \\
		(p-q)\int_0^x W^{(p)}(x-y) Z^{(q)}(y,\vartheta) \mathrm{d} y&=Z^{(p)}(x;\vartheta)-Z^{(q)}(x,\vartheta) .	\label{eq:LoeffenLTIdentity2} 
	\end{align}  
	
	\noindent Occupation time fluctuation identities rely on some more general scale functions, see  \cite{LRZ2011, LRZ2014}, namely, for $p, p+q \geq 0$ and $u, x \in \mathbb{R}$, we  define 	
		\begin{align} 
			\overline{W}_u^{(p, q)}(x)  := & \; W^{(p+q)}(x)-q \int_0^u W^{(p+q)}(x-y) W^{(p)}(y) \mathrm{d} y=   W^{(p)}(x)+q \int_u^x W^{(p+q)}(x-y) W^{(p)}(y) \mathrm{d} y,  \label{eq:SecondGenScaleFunc1}		
			\\
			\label{eq:SecondGenScaleFunc2}
			\overline{Z}_u^{(p, q)}(x)  := & \;  Z^{(p+q)}(x)-q \int_0^u W^{(p+q)}(x-y) Z^{(p)}(y) \mathrm{d} y 
			=    Z^{(p)}(x)+q \int_u^x W^{(p+q)}(x-y) Z^{(p)}(y) \mathrm{d} y. 
		\end{align}  
		Based on the above scale functions, the two-sided exit problem and the potential measure of an SNLP have the following forms. Let $\tau_{a,Z}^{+(-)}:=\inf \left\{t>0: Z_t>(<) \; a\right\}$,
		for  any process  $Z=\{Z_t\}_{t\geq 0}$ (by convention $\inf \varnothing = \infty$).  $Z$ may change in the rest of the paper, depending  on the underlying process used, without otherwise altering the notion of the stopping times.     Then, for  $x \leq a$, $a>0$  and $ q \geq 0$ (see Chapter 8 of \cite{K2014}), 
		\begin{align}
			&	\mathbb{E}_x\left(\mathrm{e}^{-q \tau_{a,X}^{+}} \mathbf{1}_{\left\{\tau_{a,X}^{+}<\tau_{0,X}^{-}\right\}}\right)=\frac{W^{(q)}(x)}{W^{(q)}(a)}, \quad  \mathbb{E}_x\left(\mathrm{e}^{-q \tau_{0,X}^{-}} \mathbf{1}_{\left\{\tau_{0,X}^{-}<\tau_{a,X}^{+}\right\}}\right)=Z^{(q)}(x) - \frac{W^{(q)}(x)}{W^{(q)}(a)}Z^{(q)}(a), \label{eq:ClassicalExitfromAbove}\\
			&	\mathbb{E}_x \Bigl( \int_0^{\infty} \mathrm{e}^{-q t} \1_{\{X_t \in \mathrm{d} y, t<\tau_{a,X}^{+} \wedge \tau_{0,X}^{-}\}} \mathrm{d} t \Bigr) = \Bigl( \frac{W^{(q)}(x) }{W^{(q)}(a)}W^{(q)}(a-y)-W^{(q)}(x-y) \Bigr) \, \mathrm{d} y, \quad  x, y \in[0, a].  \label{eq:ClassicalKilledPotential}
		\end{align}
		
		\noindent \textbf{Reflected SNLP}. The \emph{SNLP reflected at the lower boundary} 0,  is a strong Markov process defined by 
		\begin{equation}
			\wt{X}_t := X_t + \sup_{0 \leq s \leq t}\left(-X_s\right) \vee 0, \quad t \geq 0, \notag \label{eq:Regulator-ReflectionfromBelow}
		\end{equation}
		and has values on $[0,\infty)$ since the  supremum pushes the process upwards whenever it attempts to down-cross 0, see also in \cite{B1996, K2014}. A typical  application of the  reflected SNLP is for the so-called bail-out model, in which capital is injected to prevent the surplus of a firm from going below 0. 
		
		For the fluctuation identities of $\widetilde{X}$, we have (see Theorem 1 (i) and Proposition 2 (i) of \cite{P2004}) for $q,b \geq 0$, any Borel set $B$ and $x,y \in [0,b]$,
		\begin{equation}\label{eq:ReflectedBelowExitfromAbove}
			\mathbb{E}_x\left(e^{-q \tau_{b,\widetilde{X}}^{+}} \1_{\{\tau_{b,\widetilde{X}}^{+} < \infty \}}\right)=\frac{Z^{(q)}(x)}{Z^{(q)}(b)}, 
			\; \; \; 
			\mathbb{E}_x\Bigl(\int_0^{\tau_{b,\widetilde{X}}^{+}} e^{-q t} \1_{\{\wt{X}_t \in B\}} \mathrm{d} t\Bigr)= \int_B \Bigl( \frac{Z^{(q)}(x)}{Z^{(q)}(b)} {W}^{(q)}(b-y)-{W}^{(q)}(x-y) \Bigr) \dd y.  
		\end{equation}
		For a complete introduction to reflected L\'evy processes, the reader is pointed to Chapter 6 in \cite{B1996} and Chapter 8 in \cite{K2014}.
		
		\vspace{2ex}
		
		\noindent \textbf{Refracted SNLP}. The \emph{refracted SNLP} is a variant of an SNLP reflected from above, and was first introduced in \cite{KL2010}. Informally speaking, a linear drift at rate $\delta>0$ is subtracted from the increments of the underlying SNLP $X$ whenever it exceeds a pre-specified (fixed) positive level $b>0$.  Due to this characteristic, the refracted SNLP is often used to model dividend problems, where dividends are paid on a constant rate $\delta>0$, see for e.g.~\cite{JMPY2019, KLP2012}. More formally, it is the unique strong solution to the stochastic differential equation given by
		\begin{equation*}
			V_t = X_t-\delta \int^t_0 1_{\left\{V_s>b\right\}} \mathrm{d} s, \quad t \geq 0. \label{eq:RefractedLevySDE}
		\end{equation*}
		When deriving its fluctuation identities, it is important that the \emph{drift changed process} 
		$		Y_t:=X_t-\delta t$,  $t \geq 0$, 
		is again an SNLP that is not the negative of a subordinator. Hence, in \cite{KLP2012} and \cite{KL2010}, the standing assumption 
		\begin{enumerate}
			\item[] \textbf{(H)} $\quad 0< \delta<c, \; c=\mu+\int_{(-1,0)} |x|\nu(\dd x),\quad$ if $X$ has paths of bounded variation, 
		\end{enumerate}
		is imposed to show that  a unique strong solution of Eq.~\eqref{eq:RefractedLevySDE} exists, for which the unbounded variation solution is solved by using so-called strong approximation (see Theorem 1 in \cite{KL2010} for details).  
		We will denote the scale functions of $Y$ as $\mathbb{W}^{(p)}$ and $\mathbb{Z}^{(p)}$ for each $p \geq 0$ ($\overline{\W}^{(p,q)}$ and $\overline{\Z}^{(p,q)}$ for $p+q \geq 0$) which are to be interpreted as the counterparts of the scale functions ${W}^{(p)}$ and ${Z}^{(p)}$ (resp.~$\overline{W}^{(p,q)}$ and $\overline{Z}^{(p,q)}$) associated with $X$. Observe also that $\W^{(q)}(0) = (c-\delta)^{-1}$ for $Y$ of bounded variation and $\W^{(q)}(0) = 0$ for $Y$ of unbounded variation. Additionally, if $\delta = 0$, then $Y = X$ which yields that $\W^{(p)} = W^{(p)}$ ($\overline{\W}^{(p,q)} = \overline{W}^{(p,q)}$) and similarly for $\Z^{(p)}$ ($\overline{\Z}^{(p,q)}$). 
	Furthermore, for $q \geq 0$, the Laplace exponent of $Y$ will be denoted as $\psi_q^*(\vartheta):=\psi(\vartheta)-\delta \vartheta -q $, with a corresponding right-inverse $\varphi_q=\sup \left\{\vartheta \geq 0: \psi_0^*(\vartheta)=q\right\}$. Then, from \cite{KL2010}, we have the following fluctuation identities for the refracted L\'evy process.
		For $a,q \geq 0$, any Borel set $B$, $x,b \in [0,a]$ and $y \geq 0$,
		\begin{equation}
			\mathbb{E}_x\Bigl(\int_0^{\infty} e^{-q t} 1_{\left\{V_t \in B, t < \tau_{0,V}^{-} \wedge \tau_{a,V}^{+}\right\}} \mathrm{d} t\Bigr) = \int_B \Bigl( \frac{w^{(q)}(x)}{w^{(q)}(a)} w^{(q)}(a;y) -  w^{(q)}(x;y)\Bigr) \dd y, \label{eq:Refracted-KilledPotential}
		\end{equation}
		where 
		\begin{equation}
			w^{(q)}(x;y) :=  \Bigl(W^{(q)}(x - y) + \delta \int^x_b \W^{(q)}(x-u) W^{(q)\prime}(u-y) \dd u \Bigr) \1_{\{y \in [0,b)\}} + \W^{(q)}(x-y) \1_{\{y \in [b,\infty)\}},  \label{eq:RefractedWScaleFunc}
		\end{equation}
		and for which we have used the convention that $w^{(q)}(x) := w^{(q)}(x;0)$. Moreover,
		\begin{equation}
			\mathbb{E}_x\Bigl( e^{-q \tau_{a,V}^{+}} 1_{ \left\{\tau_{a,V}^{+} < \tau_{0,V}^{-} \right\}}\Bigr) = \frac{w^{(q)}(x)}{w^{(q)}(a)},
			\quad \mathbb{E}_x\Bigl( e^{-q \tau_{0,V}^{-}} 1_{ \left\{\tau_{0,V}^{-} < \tau_{a,V}^{+} \right\}}\Bigr) = z^{(q)}(x) - z^{(q)}(a) \frac{w^{(q)}(x)}{w^{(q)}(a)} , \notag  \notag
		\end{equation}
		where
		\begin{equation}
			z^{(q)}(x) := Z^{(q)}(x) + \delta q \int^x_b \W^{(q)}(x-u) W^{(q)}(u) \dd u. \label{eq:RefractedZScaleFunc}
		\end{equation}
		The refracted SNLP produces a collection of useful identities with regards to the convolutions of the scale functions $w^{(p)}$ and $z^{(p)}$ and    can be seen as generalisations of the identities from Eqs.~\eqref{eq:LoeffenLTIdentity1} and \eqref{eq:LoeffenLTIdentity2} but w.r.t.~the refracted process $V$. These, along with some fluctuation identities for the  downward crossings of the refracted process that will be used in the forthcoming, are presented in the following lemma (its  proof can be found in the Appendix). 
		\begin{lemma}\label{lem:RefractedConvolutionIdentities}
			Let $a \geq b \geq 0$. Then, for $p, p+q \geq 0$ and $x \in [0,a]$, we have the following:
			\begin{itemize}
				\item[\upshape{(i)}] For $u \in [0,b)$, \begin{equation}
					q \int_0^x \mathbb{W}^{(p+q)}(x-y) w^{(p)}(y;u) \mathrm{d} y = 	\W^{(p+q)}(x-u) - w^{(p)}(x;u) - \delta \int_{[0,b-u)} \W^{(p+q)}(x-u-y) W^{(p)} (\dd y), \label{eq:RefractedWConvolutionIdentity}
				\end{equation}
				where $W^{(p)}(\dd y)$ is the measure defined on $[0,\infty)$ associated with $W^{(p)}(a,b]:= W^{(p)}(b) - W^{(p)}(a)$ for $p \geq 0$ and $-\infty < a \leq b < \infty$ (see Chapter 8 of \cite{K2014} for more details).
				\item[\upshape{(ii)}] 	\begin{equation}
					q \int_0^x \mathbb{W}^{(p+q)}(x-y) z^{(p)}(y) \mathrm{d} y = 	\Z^{(p+q)}(x) - z^{(p)}(x) - \delta p \int^b_0 \W^{(p+q)}(x-y) W^{(p)} (y) \dd y. \label{eq:RefractedZConvolutionIdentity}
				\end{equation}
				
				\item[\upshape{(iii)}] Let $u \geq 0$ and define 
				\begin{equation}
					\ovl{w}_b^{(p,q)}(x;u) :=  w^{(p)}(x;u) +q \int_b^x \mathbb{W}^{(p+q)}(x-y) w^{(p)}(y;u) \mathrm{d} y, \label{eq:SecondGenFunc-RefractedWTilde1}
				\end{equation}
				with the convention that $\ovl{w}_b^{(p,q)}(x) = \ovl{w}_b^{(p,q)}(x;0)$. Then, we have also that
				\begin{equation}
					\begin{aligned}
						\ovl{w}_b^{(p,q)}(x;u) &= \W^{(p+q)}(x-u) -q \int_0^{b-u} \mathbb{W}^{(p+q)}(x-u-y) W^{(p)}(y) \mathrm{d} y \\
						&\eqsp - \delta \int_{[0,b-u)} \W^{(p+q)}(x-u-y) W^{(p)} (\dd y).
					\end{aligned} \label{eq:SecondGenFunc-RefractedWTilde2}
				\end{equation}
				
				\item[\upshape{(iv)}] Define
				\begin{equation}
					\ovl{z}_b^{(p,q)}(x) := z^{(p)}(x)+q \int_b^x \mathbb{W}^{(p+q)}(x-y) z^{(p)}(y) \mathrm{d} y. \label{eq:SecondGenFunc-RefractedZTilde1}
				\end{equation}
				Then, we have also that
				\begin{equation}
					\ovl{z}_b^{(p,q)}(x) =  \Z^{(p+q)}(x) -q \int_0^b \mathbb{W}^{(p+q)}(x-y) Z^{(p)}(y) \mathrm{d} y - \delta p \int^b_0 \W^{(p+q)}(x-y) W^{(p)} (y) \dd y. \label{eq:SecondGenFunc-RefractedZTilde2}
				\end{equation}
				\item[\upshape{(v)}]
				For $u \in [0,b]$,
				\begin{equation}
					\begin{aligned}
						{\E}_x \bigl(\mathrm{e}^{-(p+q) \tau_{b,Y}^{-} }\1_{\{ \tau_{b,Y}^{-}<\tau_{a,Y}^{+}\}} W^{(p)}(Y_{\tau_{b,Y}^{-}}-u) \bigr)=\ovl{w}_b^{(p,q)}(x;u)-\frac{\mathbb{W}^{(p+q)}(x-b)}{\mathbb{W}^{(p+q)}(a-b)} \ovl{w}_b^{(p,q)}(a;u),
					\end{aligned} \label{eq:W(q)JumpDownwardsinYDynamics}
				\end{equation}
				and
				\begin{equation}
					\begin{aligned}
						{\E}_x \bigl(\mathrm{e}^{-(p+q) \tau_{b,Y}^{-} }\1_{\{ \tau_{b,Y}^{-}<\tau_{a,Y}^{+}\}} Z^{(p)}(Y_{\tau_{b,Y}^{-}}) \bigr)=\ovl{z}_b^{(p,q)}(x) -\frac{\mathbb{W}^{(p+q)}(x-b)}{\mathbb{W}^{(p+q)}(a-b)} \ovl{z}_b^{(p,q)}(a), 
					\end{aligned} \label{eq:Z(q)JumpDownwardsinYDynamics}
				\end{equation}
				where $\ovl{w}_b^{(p,q)}$ and $	\ovl{z}_b^{(p,q)}$ are given in Eqs.~\eqref{eq:SecondGenFunc-RefractedWTilde1} and \eqref{eq:SecondGenFunc-RefractedZTilde1}, respectively.
			\end{itemize}
		\end{lemma}
		\begin{remark}\upshape{
			Observe that putting $\delta = 0$ yields that $\ovl{w}_b^{(p,q)}(x) = \ovl{W}_b^{(p,q)}(x)$ and $\ovl{z}_b^{(p,q)}(x) = \ovl{Z}_b^{(p,q)}(x)$. }
		\end{remark} 
		\noindent The reader is referred to \cite{KLP2012,KL2010,   KPP2014, F2014} for further details on refracted SNLPs.
		
		\vspace{2ex}
		\noindent \textbf{Refracted-reflected  SNLP}.
		In this section, we provide a brief exposition of the \emph{refracted-reflected L\'evy} process. For the construction of the process, we refer to  \cite{PY2018}.
		
		For fixed $b>0$ and $\delta > 0$ such that the condition \textbf{(H)}  holds, we define a refracted-reflected L\'evy process $U=\{U_t\}_{t\geq 0}$, a process for which a linear drift at rate $\delta$ is subtracted from its increments when it is above the level $b$ (refracted) and which is reflected at $0$. Hence, for  $R_t$,  a non-decreasing and right continuous process that represents the cumulative amounts  up to $t$ that pushes the process upwards when it attempts to go below 0, $U$ is the solution of
		%
		%
			%
			%
			%
			%
			%
			\begin{equation}
				U_t=X_t+{R}_t-\delta \int^t_0 \1_{\{U_s > b\}} \dd s, \quad t \geq 0,
			\end{equation}
			for which, in the particular case of bounded variation, we have
			\begin{equation}
				{R}_t = \sum_{t \geq 0: U_{t-} + \Delta X_t < 0} | U_{t-} + \Delta X_t |, \quad t \geq 0. \notag
			\end{equation}
			If, in addition, $\PP_x(U_t = y) = 0$ for $y \in [0,\infty)$ and Lebesgue a.e.~$t>0$ and the condition \textbf{(H)} is satisfied, then the unbounded variation solution for $U$ can be determined by strong approximation (see Proposition 2.1 of \cite{PY2018}). Hence, for  $q \geq 0$, any Borel set $B$ and $x,y,b \in [0,a]$, we have (see Theorem 4.1 and Corollary 4.2 in \cite{PY2018})
			\begin{equation}
				\mathbb{E}_x\Bigl( e^{-q \tau_{a,U}^+} 1_{\left\{\tau_{a,U}^+ < \infty \right\}} \Bigr)=\frac{z^{(q)}(x)}{z^{(q)}(a)}, \quad
				\mathbb{E}_x\Bigl(\int_0^{\infty} e^{-q t} 1_{\{U_t \in B, \; t < \tau_{a,U}^{+} \}} \mathrm{d} t\Bigr)= \int_B \Bigl(\frac{z^{(q)}(x)}{z^{(q)}(a)} {w}^{(q)}(a;y)-{w}^{(q)}(x ; y) \Bigr) \dd y. \label{eq:ReflectedRefractedPotentialMeasure}
			\end{equation}
			where $w^{(q)}$ and $z^{(q)}$ are defined in Eqs.~\eqref{eq:RefractedWScaleFunc} and \eqref{eq:RefractedZScaleFunc}, respectively.

		\section{Poissonian potential measures with two-sided killing}
		\label{main}
		In this section, we derive the Poissonian potential measures for $U$ and denote by the constant $b$ its corresponding refraction barrier unless specified otherwise. To do this  we denote by $S^\lambda_i$,  $S^\gamma_i$, for $i \in \N$, the arrival times of two independent Poisson processes $N_1^{\lambda}$, $N_2^{\gamma}$,  with rates $\lambda > 0$ and $\gamma>0$, respectively. We also assume that  $N_1^\lambda$ and  $N_2^\gamma$ are independent from the  underlying process $U$ and that the first observation occurs at either $S^\lambda_1$ or $S^\gamma_1$. Moreover,  for any process $Z$, we define for  $a>0$ the first passage Poissonian observation times 
		\begin{equation*}
			{T}_{a,Z}^{+ (-),k}= \; \min \{S^k_i: {Z}_{S^k_i}> (<) \; a \},\quad k=\lambda, \gamma,
		\end{equation*}
		where as before  $Z$ in the subscript indicates the underlying process (by convention $\min \varnothing = \infty$). We note that $Z$ may change in the following sections to $\wt{X}$, $Y$ and $U$ depending on the process under consideration.  Clearly, it holds that $\tau_{a,X }^{+ (-)} \leq T_{a,X}^{+ (-),k}$, where $k = \lambda,\gamma$, and similar inequalities hold for the stopping times with corresponding subscripts  $\wt{X}$, $Y$, and $U$,  where for the reader's convenience, we also recall that $\{X_t\}_{t\geq0}$ denotes a SNLP, $\{Y_t\}_{t\geq0}$ with $Y_t = X_t - \delta t$ denotes the drift-changed process, $\{\wt{X}_t\}_{t \geq 0}$ denotes the SNLP reflected from below at $0$ and $\{U_t\}_{t\geq 0}$ denotes the refracted-reflected SNLP.

		Our goal is to derive the potential densities of
		\begin{align}
			&\mathbb{E}_x \Bigl(\int_0^{\infty} \ee^{-q t} 1_{\left\{U_t \in \dd u, t < T_{b,U}^{-,\gamma} \wedge T_{a,U}^{+,\lambda} \right\}} \mathrm{d} t\Bigr), \quad \quad  \mathbb{E}_x \Bigl(\int_0^{\infty} \ee^{-q t} 1_{\left\{U_t \in \dd u, t < T_{b,U}^{+,\gamma} \wedge T_{a,U}^{+,\lambda} \right\}} \mathrm{d} t\Bigr), \notag \\
			&	\mathbb{E}_x \Bigl(\int_0^{\infty} \ee^{-q t} 1_{\left\{U_t \in \dd u, t < T_{b,U}^{+,\gamma} \wedge T_{a,U}^{-,\lambda} \right\}} \mathrm{d} t\Bigr), \quad \quad \mathbb{E}_x \Bigl(\int_0^{\infty} \ee^{-q t} 1_{\left\{U_t \in \dd u, t < T_{b,U}^{-,\gamma} \wedge T_{a,U}^{-,\lambda} \right\}} \mathrm{d} t\Bigr), \notag
		\end{align}
		and to take appropriate limits of $\gamma$ and $\lambda$ to derive further identities, see Section \ref{Sec:LimitsofMainTheorems} for more details. 
	The quantities above consider that each observer is interested in the process  behaviour above and below one of the model barriers respectively, i.e. the barrier $b$ for the $\gamma$-rate observer and barrier $a$ for the $\lambda$-rate observer. We are interested in these particular quantities due to their potential applications in insurance and other areas. For example, in the context of insurance, the $\gamma$-rate observer may be concerned with payments of dividends which occur based on the level $b$ and the $\lambda$-observer with excess capital levels (above level $a$) which could be invested or transferred to other business lines.  
		
		\noindent To derive them, we will require the next two lemmas, whose proofs can be found in the Appendix.
		
		\begin{lemma} \label{lem:PoissonPotentialsAndFluctuations}
			Let $0 \leq b \leq a$, $q \geq 0$ and $0 < \lambda < \infty$. Then the following identities hold:
			\begin{enumerate}
				\item[\upshape{(i)}] For $ x \in [b,a]$ and $ y \in [b,\infty)$,
				\begin{equation}
					\E_x\Bigl( \int^\infty_0 e^{-q t} \1_{\{Y_t \in \dd y, \; t < T_{a,Y}^{+,\lambda} \wedge \tau_{b,Y}^- \}} \dd t \Bigr) = \Bigl( \frac{\W^{(q)}(x-b)}{\Z^{(q)}(a-b, \varphi_{q+\lambda})} \Z^{(q)}(a-y, \varphi_{q+\lambda}) - \W^{(q)}(x-y)\Bigr) \dd y, \label{eq:Landriault-Identity1}
				\end{equation}
				and, consequently, for $ y \in (a,\infty)$,	
				\begin{equation}
					\E_x\bigl( \ee^{-q T_{a,Y}^{+,\lambda}} \1_{\{Y_{T_{a,Y}^{+,\lambda}} \in \dd y, \; T_{a,Y}^{+,\lambda} < \tau_{b,Y}^-\}} \bigr) = \lambda \frac{\W^{(q)}(x-b)}{\Z^{(q)}(a-b,\varphi_{q+\lambda})} \ee^{\varphi_{q+\lambda}(a-y)} \dd y.  \label{eq:Landriault-Identity2}
				\end{equation}
				
				\item [\upshape{(ii)}] For $ x,y \in [b,\infty)$,	
				\begin{equation}
					\E_x \Bigl( \int^\infty_0 \ee^{-q t} \1_{\{Y_t \in \dd y, \; t < T_{a,Y}^{-,\lambda} \wedge \tau_{b,Y}^- \}} \dd t \Bigr) = \Bigl(\frac{\Z^{(q+\lambda)}(a-y, \varphi_{q})}{\Z^{(q+\lambda)}(a-b, \varphi_{q})} \W^{(q+\lambda)}(x-b) - \W^{(q+\lambda)}(x-y) \Bigr) \dd y, \label{eq:Landriault-Identity3}
				\end{equation}
				and consequently, for $y \in (b,a)$,
				\begin{equation}
					\E_x \Bigl( \ee^{-q T_{a,Y}^{-,\lambda}} \1_{\{Y_{T_{a,Y}^{-,\lambda}} \in \dd y,\; T_{a,Y}^{-,\lambda} < \tau_{b,Y}^- \}} \Bigr) = \lambda \Bigl(\frac{\Z^{(q+\lambda)}(a-y, \varphi_{q})}{\Z^{(q+\lambda)}(a-b, \varphi_{q})} \W^{(q+\lambda)}(x-b) - \W^{(q+\lambda)}(x-y) \Bigr) \dd y. \label{eq:Landriault-Identity4}
				\end{equation}
			\end{enumerate}
		\end{lemma}
		
		\noindent Lastly, we define the auxiliary functions
		\begin{align}
			\alpha_1^{(q,p)}(b;u) 
			& := \ee^{-\varphi_q u} \Bigl( 1 + (p-q)\int_0^{b-u} \ee^{-\varphi_q y}  W^{(p)}(y) \mathrm{d} y - \delta \int_{[0,b-u)} \ee^{-\varphi_q y} W^{(p)} (\dd y)  \Bigr), \label{eq:LimitOfWBar} \\
			\alpha_2^{(q,p)}(b) 
			&:= \frac{q}{\varphi_q} + (p-q)\int_0^{b} \ee^{-\varphi_q y}  Z^{(p)}(y) \mathrm{d} y - \delta p \int_0^b \ee^{-\varphi_q y} W^{(p)} (y) \dd y , \label{eq:LimitOfZBar}
		\end{align}
		and have the following. 
		\begin{lemma}\label{lem:SimilartoRonnieLemma}
			Let $\lambda > 0$, $p,q \geq 0$, $x \in [b,a]$ and $u \in [0,b)$. Further, define $f^{(p)}_{1,1}(x) := W^{(p)}(x-u)$, $f^{(p)}_{1,2}(x) := Z^{(p)}(x)$, $f^{(p,q)}_{2,1}(x) := \ovl{w}_b^{(p,q-p)}(x;u)$ and $f^{(p,q)}_{2,2}(x) := \ovl{z}_b^{(p,q-p)}(x)$. Then, for $\mathrm{j} =1,2$, it holds that 
			\begin{align}
				\text{\upshape{(i)}}&\quad\E_x\bigl( \ee^{-q\tau_{b,Y}^-} \1_{\{\tau_{b,Y}^- < T_{a,Y}^{+,\lambda}\}} f^{(p)}_{1,\jj}(Y_{\tau_{b,Y}^-}) \bigr) = f^{(p,q)}_{2,\jj}(x) - \frac{\W^{(q)}(x-b)}{\Z^{(q)}(a-b,\varphi_{q+\lambda})}\lambda \int^\infty_0 \ee^{-\varphi_{q+\lambda}y}f^{(p,q)}_{2,\jj}(y+a)\dd y,  \label{eq:MainLemmaResultLikeRonnie} \\
				\text{\upshape{(ii)}}&\quad \E_x\bigl( \ee^{-q\tau_{b,Y}^-} \1_{\{\tau_{b,Y}^- < T_{a,Y}^{-,\lambda}\}} f^{(p)}_{1,\jj}(Y_{\tau_{b,Y}^-}) \bigr) = f^{(p,q)}_{2,\jj}(x) + \lambda \int^x_b \W^{(q+\lambda)}(x-y)f^{(p,q)}_{2,\jj}(y)\dd y \notag \\
				&\quad \quad \quad - \frac{\W^{(q+\lambda)}(x-b)}{\Z^{(q+\lambda)}(a-b,\varphi_{q})}\Bigl( c_\jj^{(q,p)} e^{\varphi_q (a-b)} + \lambda \int^a_b \Z^{(q+\lambda)}(a-y,\varphi_q) f_{2,\jj}^{(p,q)}(y) \dd y \Bigr), \label{eq:MainLemmaResultKindaLikeRonnie}
			\end{align}
			where we have denoted $c_1^{(q,p)} := \ee^{\varphi_q b} \alpha_1^{(q,p)}(b;u)$ and $c_2^{(q,p)} := \ee^{\varphi_q b} \alpha_2^{(q,p)}(b)$.
		\end{lemma}

		\begin{remark} \label{rem:InitialValueTheorem}\upshape{
			\begin{itemize}
				\item[\upshape(i)] 	By using the initial value theorem for Laplace transforms (see, for instance, Theorem 2.34 p.~88 in \cite{S1999}), we get
				\begin{equation}
					\lim_{\lambda \rightarrow \infty} \frac{\varphi_{q+\lambda}}{\lambda}\Z^{(q)}(a-b,\varphi_{q+\lambda}) = \W^{(q)}(a-b), \quad \text{ and } \quad
					\lim_{\lambda \rightarrow \infty} \varphi_{q+\lambda} \int^\infty_0 \ee^{-\varphi_{q+\lambda}y}f^{(p,q)}_{2,\jj}(y+a)\dd y = f^{(p,q)}_{2,\jj}(a), \notag
				\end{equation}
				for $\jj \in \{1,2\}$. Thus, by using the above equations along with Eq.~\eqref{eq:MainLemmaResultLikeRonnie} and dominated convergence,
				\begin{equation}
					\E_x\bigl( \ee^{-q\tau_{b,Y}^-} \1_{\{\tau_{b,Y}^- < \tau_{a,Y}^+\}} f^{(p)}_{1,\jj}(Y_{\tau_{b,Y}^- }) \bigr) = \lim_{\lambda \rightarrow \infty} \E_x\bigl( \ee^{-q\tau_{b,Y}^-} \1_{\{\tau_{b,Y}^- < T_{a,Y}^{+,\lambda}\}} f^{(p)}_{1,\jj}(Y_{\tau_{b,Y}^- }) \bigr) = f^{(p,q)}_{2,\jj}(x) - \frac{\W^{(q)}(x-b)}{\W^{(q)}(a-b)}f^{(p,q)}_{2,\jj}(a),  \notag
				\end{equation}
				which coincides with Eqs.~\eqref{eq:W(q)JumpDownwardsinYDynamics} and \eqref{eq:Z(q)JumpDownwardsinYDynamics}, as expected.
				
				\item[\upshape{(ii)}] The requirements of the initial value theorem (see Theorem 2.7 p.~54 in \cite{S1999}) are that the underlying functions are sufficiently smooth and of  exponential order which are satisfied in all our applications of it. Indeed, we need only consider the underlying scale functions $W$ and $\W$ since all other scale functions considered are compositions thereof. It is well known (see Lemma 2.3 of \cite{KKR2012} as well as Section 3.5 therein for further discussions on smoothness) that the scale functions $W$ and $\W$ are almost everywhere differentiable, and furthermore from Chapter 8 of \cite{K2014} that $W^{(q)}(x) = \ee^{\Phi_q x} W_{\Phi_q}(x)$ and $\W^{(q)}(x) = \ee^{\varphi_q x} \W_{\varphi_q}(x)$ which are indeed of exponential order. 
			\end{itemize}}
		\end{remark}
		\noindent We now state and prove our main theorems.
		
		\begin{theorem}\label{Thm:Main-ReflectedRefractedDoubleKilledPoissonPotential}
			Let $q \geq 0$. Then, for $\lambda, \gamma > 0$, $x,b \in [0,a]$ and $u \in [0,\infty)$,
			\begin{align}
				&\mathbb{E}_x \Bigl(\int_0^{\infty} e^{-q t} 1_{\left\{U_t \in \dd u, t < T_{b,U}^{-,\gamma} \wedge T_{a,U}^{+,\lambda} \right\}} \mathrm{d} t\Bigr) \big/ \dd u \notag \\
				&= \frac{\lambda \int^\infty_0 \ee^{-\varphi_{q+\lambda}y}\overline{w}_b^{(q+\gamma,-\gamma)}(y+a;u) \dd y \1_{\{u \in [0,a]\}} + \ee^{-\varphi_{q+\lambda}(u-a)}\1_{\{u \in (a,\infty)\}}}{\lambda \int^\infty_0 \ee^{-\varphi_{q+\lambda}y}\overline{z}_b^{(q+\gamma,-\gamma)}(y+a) \dd y} \overline{z}_b^{(q+\gamma,-\gamma)}(x) -  \overline{w}_b^{(q+\gamma,-\gamma)}(x;u), \label{eq:ReflectedRefracted-DoubleKilledPotential}
			\end{align}
			where $\overline{w}_b^{(\cdot,\cdot)}$ and $\overline{z}_b^{(\cdot,\cdot)}$ are defined in Eqs.~\eqref{eq:SecondGenFunc-RefractedWTilde1} and \eqref{eq:SecondGenFunc-RefractedZTilde1}, respectively. 
		\end{theorem}
		
		%
			\begin{remark}\label{Rem:AboutHowFormsAreTheSame}\upshape{
				The identity $\mathbb{E}_x \bigl(\int_0^{\infty} \ee^{-q t} \1_{\left\{U_t \in \dd u, t < T_{b,U}^{+,\gamma} \wedge T_{a,U}^{+,\lambda} \right\}} \mathrm{d} t\bigr) \big/ \dd u$ can be proven by using similar arguments as in the proof below  with the necessary adjustments. The form of this identity can be obtained by replacing all instances of $q$ and $\gamma$ in the r.h.s.~of Eq.~\eqref{eq:ReflectedRefracted-DoubleKilledPotential} with $q+\gamma$ and $-\gamma$, respectively. To be more precise, we have that instances of $\varphi_{q+\lambda}$, $\overline{w}_b^{(q+\gamma,-\gamma)}$ and $\overline{z}_b^{(q+\gamma,-\gamma)}$ are to be replaced by $\varphi_{q +\gamma+\lambda}$, $\overline{w}_b^{(q,\gamma)}$ and $\overline{z}_b^{(q,\gamma)}$, respectively, in the r.h.s.~of Eq.~\eqref{eq:ReflectedRefracted-DoubleKilledPotential} to obtain the desired Poissonian potential measure.}
			\end{remark}

			\begin{proof}[Proof of Theorem \ref{Thm:Main-ReflectedRefractedDoubleKilledPoissonPotential}]
				We consider first the bounded variation case. Let  $e_\gamma \sim \text{Exp}(\gamma)$ be independent of all other random variables and define  $R^{(q)}_{\lambda,\gamma}(x,\dd u) :=\; \E_x\bigl( \int^\infty_0 e^{-q t} 1_{\{U_t \in \dd u, t < T_{b,U}^{-,\gamma} \wedge T_{a,U}^{+,\lambda} \}} \dd t \bigr)$. Recall that $\{U_t, t < \tau_{b,U}^+\} \overset{\dd}{=} \{\wt{X}_t, t < 
				\tau_{b,\wt{X}}^+\}$ w.r.t. $\PP_x$ for $x \in [0,b)$ and, moreover, observe that $T_{b,U}^{-,\gamma} \wedge T_{a,U}^{+,\lambda} \wedge \tau_{b,U}^+ =_\dd e_\gamma \wedge \tau_{b,U}^+$. Hence, by conditioning on $\tau_{b,U}^+$ and using the strong Markov property, it follows that 
				\begin{align}
					R^{(q)}_{\lambda,\gamma}(x,\dd u)
					=& \; \E_x\left( \int^\infty_0 \ee^{-q t} \1_{\{U_t \in \dd u, \; t < T_{b,U}^{-,\gamma} \wedge \tau_{b,U}^{+} \}} \dd t \right) + \E_x\left( \ee^{-q \tau_{b,U}^{+}} \1_{\{\tau_{b,U}^{+} < T_{b,U}^{-,\gamma}\}} \right) R^{(q)}_{\lambda,\gamma}(b, \dd u) \notag \\
					=& \; \E_x\left( \int^\infty_0 \ee^{-q t} \1_{\{U_t \in \dd u, \; t < e_{\gamma} \wedge \tau_{b,U}^{+} \}} \dd t \right) + \E_x\left( \ee^{-q \tau_{b,U}^{+}} \1_{\{\tau_{b,U}^{+} < e_\gamma \}} \right) R^{(q)}_{\lambda,\gamma}(b, \dd u) \notag \\
					=& \; \E_x\left( \int^\infty_0 \ee^{-(q+\gamma) t} \1_{\{\wt{X}_t \in \dd u, \; t < \tau_{b,\wt{X}}^{+} \}} \dd t \right) + \E_x\left( \ee^{-(q+\gamma) \tau_{b,\wt{X}}^{+}} \1_{\{\tau_{b,\wt{X}}^{+} < \infty \}} \right) R^{(q)}_{\lambda,\gamma}(b, \dd u)  \notag \\
					=& \; \frac{Z^{(q+\gamma)}(x)}{Z^{(q+\gamma)}(b)} \Bigl( {W}^{(q+\gamma)}(b-u) \1_{\{u \in [0,b)\}} \dd u + R^{(q)}_{\lambda,\gamma}(b, \dd u) \Bigr) - {W}^{(q+\gamma)}(x-u) \1_{\{u \in [0,b)\}} \dd u, \label{eq:Thm-SMPx<bIdentity1}
				\end{align}
				where the last equality follows by using the identities in Eq.~\eqref{eq:ReflectedBelowExitfromAbove}.
				
				Now, considering $x \in [b,a]$ and noticing that $\{U_t, t < \tau_{b,U}^-\} \overset{\dd}{=} \{Y_t, t < \tau_{b,Y}^-\}$ w.r.t.~$\PP_x$ for these $x$-values, we condition on $\tau_{b,Y}^-$ and use the strong Markov property to get
				\begin{align}
					R^{(q)}_{\lambda,\gamma}(x,\dd u) 
					=& \; \E_x\left( \int^\infty_0 e^{-q t} \1_{\{Y_t \in \dd u, \; t < \tau_{b,Y}^- \wedge T_{a,Y}^{+,\lambda} \}} \dd t \right) + \E_x\Bigl( e^{-q \tau_{b,Y}^-} \1_{ \{ \tau_{b,Y}^- < T_{a,Y}^{+,\lambda} \}} R^{(q)}_{\lambda,\gamma}(Y_{\tau_{b,Y}^-},\dd u) \Bigr) \notag \\
					=& \; \Bigl( \frac{\W^{(q)}(x-b)}{\Z^{(q)}(a-b, \varphi_{q+\lambda})} \Z^{(q)}(a-u, \varphi_{q+\lambda}) - \W^{(q)}(x-u)\Bigr) \1_{\{u \in [b,\infty)\}}  \dd u \notag \\
					&+ \frac{1}{Z^{(q+\gamma)}(b)} \Bigl( {W}^{(q+\gamma)}(b-u) \1_{\{u \in [0,b)\}} \dd u + R^{(q)}_{\lambda,\gamma}(b, \dd u) \Bigr) \; \E_x\Bigl( e^{-q \tau_{b,Y}^-} \1_{ \{ \tau_{b,Y}^- < T_{a,Y}^{+,\lambda} \}} Z^{(q+\gamma)}(Y_{\tau_{b,Y}^-}) \Bigr) \notag \\
					&- \E_x\Bigl( e^{-q \tau_{b,Y}^-} \1_{ \{ \tau_{b,Y}^- < T_{a,Y}^{+,\lambda} \}} W^{(q+\gamma)}(Y_{\tau_{b,Y}^-} - u) \Bigr) \1_{\{u \in [0,b)\}} \dd u , \notag
				\end{align}
			where, for the last equality, the first term follows by using Eq.~\eqref{eq:Landriault-Identity1} and the remaining terms follow by a substitution of Eq.~\eqref{eq:Thm-SMPx<bIdentity1}.
		
		We observe for the equation above that the first and second expectation terms can be evaluated by using Lemma \ref{lem:SimilartoRonnieLemma} (i) which yields
		\begin{align}
			R^{(q)}_{\lambda,\gamma}(x,\dd u) &= \; \Bigl( \frac{\W^{(q)}(x-b) }{\Z^{(q)}(a-b, \varphi_{q+\lambda})} \Z^{(q)}(a-u, \varphi_{q+\lambda})- \W^{(q)}(x-u)\Bigr) \1_{\{u \in [b,\infty)\}} \dd u \notag \\
			& \eqsp +  \frac{1}{Z^{(q+\gamma)}(b)} \Bigl( {W}^{(q+\gamma)}(b-u) \1_{\{u \in [0,b)\}} \dd u + R^{(q)}_{\lambda,\gamma}(b, \dd u) \Bigr) \notag \\
			& \eqsp \times \Bigl( \overline{z}_b^{(q+\gamma,-\gamma)}(x) - \frac{\W^{(q)}(x-b) }{\Z^{(q)}(a-b, \varphi_{q+\lambda})} \lambda \int^\infty_0 \ee^{-\varphi_{q+\lambda}y}\overline{z}_b^{(q+\gamma,-\gamma)}(y+a) \dd y \Bigr) \notag \\
			& \eqsp - \Bigl( \overline{w}_b^{(q+\gamma,-\gamma)}(x;u) - \frac{\W^{(q)}(x-b) }{\Z^{(q)}(a-b, \varphi_{q+\lambda})} \lambda \int^\infty_0 \ee^{-\varphi_{q+\lambda}y} \overline{w}_b^{(q+\gamma,-\gamma)}(y+a;u) \dd y \Bigr) \1_{\{u \in [0,b)\}} \dd u \notag \\
			&= \frac{1}{Z^{(q+\gamma)}(b)} \Bigl( {W}^{(q+\gamma)}(b-u) \1_{\{u \in [0,b)\}} \dd u + R^{(q)}_{\lambda,\gamma}(b, \dd u) \Bigr) \notag \\
			& \eqsp \times \Bigl( \overline{z}_b^{(q+\gamma,-\gamma)}(x) - \frac{\W^{(q)}(x-b) }{\Z^{(q)}(a-b, \varphi_{q+\lambda})} \lambda \int^\infty_0 \ee^{-\varphi_{q+\lambda}y}\overline{z}_b^{(q+\gamma,-\gamma)}(y+a) \dd y \Bigr) \notag \\
			& \eqsp - \Bigl( \Bigl[ \overline{w}_b^{(q+\gamma,-\gamma)}(x;u)\1_{\{u \in [0,b)\}} + \W^{(q)}(x-u)\1_{\{u \in [b,\infty)\}}   \Bigr] \notag \\
			& \eqsp - \frac{\W^{(q)}(x-b) }{\Z^{(q)}(a-b, \varphi_{q+\lambda})} \Bigl[ \lambda \int^\infty_0 \ee^{-\varphi_{q+\lambda}y} \overline{w}_b^{(q+\gamma,-\gamma)}(y+a;u) \1_{\{u \in [0,b)\}} \dd y  \notag \\
			&\eqsp + \Z^{(q)}(a-u, \varphi_{q+\lambda})\1_{\{u \in [b,\infty)\}} \Bigr] \Bigr) \dd u. \label{Thm-eq:AlmostLastRenewalEq}
		\end{align}
			Considering the terms of the above  equation separately so that we may simplify them, we notice from Eqs.~\eqref{eq:SecondGenFunc-RefractedWTilde1} and \eqref{eq:SecondGenFunc-RefractedWTilde2} that 
		\begin{align}
			\overline{w}_b^{(q+\gamma,-\gamma)}(x;u) \dd u 
			&= \overline{w}_b^{(q+\gamma,-\gamma)}(x;u)\1_{\{u \in [0,b)\}} \dd u +  \W^{(q)}(x-u)\1_{\{u \in [b,\infty)\}} \dd u, \label{eq:WBarUsefulIdentity1}
		\end{align}
		and from Eq.~\eqref{eq:GeneralisedZFunc} that 
		\begin{equation}
			\Z^{(q)}(a-u,\varphi_{q+\lambda})\1_{\{u \in [b,\infty)\}} \dd u = \; \Bigl( \lambda \int_0^\infty \ee^{-\varphi_{q+\lambda} y} \W^{(q)}(y + a - u) \dd y \1_{\{u \in [b,a]\}} + \ee^{-\varphi_{q+\lambda}(u-a)} \1_{\{u \in (a,\infty)\}} \Bigr) \dd u. \notag
		\end{equation}
		Hence, the last term of Eq.~\eqref{Thm-eq:AlmostLastRenewalEq} becomes
		\begin{align}
			& \frac{\W^{(q)}(x-b) }{\Z^{(q)}(a-b, \varphi_{q+\lambda})} \Bigl( \lambda \int^\infty_0 \ee^{-\varphi_{q+\lambda}y} \overline{w}_b^{(q+\gamma,-\gamma)}(y+a;u) \1_{\{u \in [0,b)\}} \dd y + \Z^{(q)}(a-u, \varphi_{q+\lambda})\1_{\{u \in [b,\infty)\}}  \Bigr) \dd u \notag \\
			&= \frac{\W^{(q)}(x-b) }{\Z^{(q)}(a-b, \varphi_{q+\lambda})}  \Bigl(\lambda \int_0^\infty \ee^{-\varphi_{q+\lambda} y}  \overline{w}_b^{(q+\gamma,-\gamma)}(y+a;u) \1_{\{u \in [0,a]\}} \dd y + \ee^{-\varphi_{q+\lambda}(u-a)} \1_{\{u \in (a,\infty)\}} \Bigr) \dd u. \notag
		\end{align}
			Therefore, using the above three equations,  Eq.~\eqref{Thm-eq:AlmostLastRenewalEq} becomes
		\begin{align}
			R^{(q)}_{\lambda,\gamma}(x,\dd u) &= \frac{1}{Z^{(q+\gamma)}(b)} \Bigl( {W}^{(q+\gamma)}(b-u) \1_{\{u \in [0,b)\}} \dd u + R^{(q)}_{\lambda,\gamma}(b, \dd u) \Bigr) \notag \\
			& \eqsp \times \Bigl( \overline{z}_b^{(q+\gamma,-\gamma)}(x) - \frac{\W^{(q)}(x-b) }{\Z^{(q)}(a-b, \varphi_{q+\lambda})} \lambda \int^\infty_0 \ee^{-\varphi_{q+\lambda}y}\overline{z}_b^{(q+\gamma,-\gamma)}(y+a) \dd y \Bigr) \notag \\
			& \eqsp - \Bigl( \overline{w}_b^{(q+\gamma,-\gamma)}(x;u) - \frac{\W^{(q)}(x-b) }{\Z^{(q)}(a-b, \varphi_{q+\lambda})} \Bigl[ \lambda \int_0^\infty \ee^{-\varphi_{q+\lambda} y} \overline{w}_b^{(q+\gamma,-\gamma)}(y + a; u) \dd y \1_{\{u \in [0,a]\}} \notag \\
			& \eqsp \eqsp + \ee^{-\varphi_{q+\lambda}(u-a)} \1_{\{u \in (a,\infty)\}} \Bigr] \Bigr) \dd u , \label{Thm-eq:AlmostLastRenewalEq2}
		\end{align}
		and hence what remains is to determine $R^{(q)}_{\lambda,\gamma}(b,\dd u)$. To do this, we note from  Eq.~\eqref{eq:RefractedWScaleFunc} and \eqref{eq:SecondGenFunc-RefractedWTilde1} that
		\begin{equation}
			\begin{aligned}
				\overline{w}_b^{(q+\gamma,-\gamma)}(b;u) \dd u &= w^{(q+\gamma)}(b;u)\1_{\{u \in [0,b)\}} \dd u + \W^{(q)}(b-u)\1_{\{u \in [b,\infty)\}} \dd u \\
				&= W^{(q+\gamma)}(b-u)\1_{\{u \in [0,b)\}} \dd u,
			\end{aligned} \notag
		\end{equation}
		where we have used that $\W^{(q)}(b-u) = 0$ for $u \in (b,\infty)$ and that $\{ u \in \{b\}\}$ has Lebesgue measure $0$. In addition, by using Eq.~\eqref{eq:SecondGenFunc-RefractedZTilde1} and then Eq.~\eqref{eq:RefractedZScaleFunc}, we have that $\overline{z}_b^{(q+\gamma,-\gamma)}(b) = Z^{(q+\gamma)}(b)$. Since $\W^{(q)}(0) \neq 0$ in the bounded variation case, letting $x=b$ in Eq.~\eqref{Thm-eq:AlmostLastRenewalEq2}, using the above two expressions for $\overline{w}_b^{(q+\gamma,-\gamma)}(b;u)$, $\overline{z}_b^{(q+\gamma,-\gamma)}(b)$ and solving w.r.t.~$R^{(q)}_{\lambda,\gamma}(b,\dd u)$ yields
			\begin{equation}
			\begin{aligned}
				&R^{(q)}_{\lambda,\gamma}(b,\dd u) \\
				&= \Bigl( \frac{\lambda \int^\infty_0 \ee^{-\varphi_{q+\lambda}y}\overline{w}_b^{(q+\gamma,-\gamma)}(y+a;u) \dd y \1_{\{u \in [0,a]\}} + \ee^{-\varphi_{q+\lambda}(u-a)}\1_{\{u \in (a,\infty)\}}}{\lambda \int^\infty_0 \ee^{-\varphi_{q+\lambda}y}\overline{z}_b^{(q+\gamma,-\gamma)}(y+a) \dd y} {Z}^{(q+\gamma)}(b) -  W^{(q+\gamma)}(b-u)\1_{\{u \in [0,b)\}}\Bigr) \dd u.
			\end{aligned} \notag
		\end{equation}
		The results for $x \in [0,b)$ and $x \in [b,a]$ then follow by substituting the above equation into Eqs.~\eqref{eq:Thm-SMPx<bIdentity1} and \eqref{Thm-eq:AlmostLastRenewalEq2}, respectively.
		To prove the unbounded variation case, we use strong approximation. 
		First recall that Proposition 2.1 in \cite{PY2018} proves that there exists a sequence of processes $\{(U_s^{(n)})_{s \geq 0}: n \geq 1\}$ that strongly approximates $U$; i.e.~that $\lim_{n \uparrow \infty} \sup_{0 \leq s \leq t}\bigl|  U_s - U_s^{(n)}\bigr| = 0$ for any $t >0$ a.s.~(see p.~210 of \cite{B1996} and  Definition 11 of \cite{KL2010} for more details).
		For $i \geq 1$, we denote $T_{b,U}^{-,\gamma}(n) := \min \{T_i^\gamma: U^{(n)}_{T_i^\gamma} < b \}$ and	$T_{a,U}^{+,\lambda}(n) := \min \{T_i^\lambda: U^{(n)}_{T_i^\lambda} > a \}$ 
		the stopping times corresponding to each process $U^{(n)}$. Then, it also holds for any time $s>0$ $\PP_x$-a.s.~(see proof of Lemma 4 (i) in \cite{BRP2026}) that $T_{b,U}^{-,\gamma}(n) \wedge s \rightarrow T_{b,U}^{-,\gamma} \wedge s$ and $T_{a,U}^{+,\lambda}(n) \wedge s \rightarrow T_{a,U}^{+,\lambda} \wedge s$ since the processes $N^\gamma$, $N^\lambda$ and $U^{(n)}$ are mutually independent for every $n \geq 1$. 
		
			Next we shall show that the potential measure of Theorem \ref{Thm:Main-ReflectedRefractedDoubleKilledPoissonPotential} (i.e.~the left-hand side of Eq.~\eqref{eq:ReflectedRefracted-DoubleKilledPotential}) of the bounded variation case converges to the one of unbounded variation. 
		By using the triangle inequality, we observe for $s >0$ that 
			\begin{align}
			&\Bigl| \int^\infty_0 \ee^{-qt} \PP_x\bigl( U^{(n)}_t \in \dd u, \; t < T_{b,U}^{-,\gamma} (n)\wedge T_{a,U}^{+,\lambda}(n) \bigr) \dd t - \int^\infty_0 \ee^{-qt} \PP_x\bigl( U_t \in \dd u, \; t < T_{b,U}^{-,\gamma} \wedge T_{a,U}^{+,\lambda} \bigr) \dd t \Bigr| \notag \\
			& \leq \int^\infty_s \ee^{-qt} \PP_x\bigl( U^{(n)}_t \in \dd u, \; t < T_{b,U}^{-,\gamma}(n) \wedge T_{a,U}^{+,\lambda}(n) \bigr) \dd t +  \int^\infty_s \ee^{-qt} \PP_x\bigl( U_t \in \dd u, \; t < T_{b,U}^{-,\gamma} \wedge T_{a,U}^{+,\lambda} \bigr) \dd t \notag \\
			& \eqsp + \int^\infty_0 \ee^{-qt} \Bigl| \PP_x\bigl( U^{(n)}_t \in \dd u, \; t < T_{b,U}^{-,\gamma}(n) \wedge T_{a,U}^{+,\lambda} (n)\wedge s \bigr) - \PP_x\bigl( U_t \in \dd u, \; t < T_{b,U}^{-,\gamma} \wedge T_{a,U}^{+,\lambda} \wedge s \bigr) \Bigr| \dd t,  \label{eq:Thm1-TriangleIneq2}
		\end{align}
		where the last inequality follows by noticing that 
		\begin{align*}
			&	\PP_x\bigl( U^{(n)}_t \in \dd u, \; t < T_{b,U}^{-,\gamma}(n) \wedge T_{a,U}^{+,\lambda}(n) \bigr) - \PP_x\bigl( U^{(n)}_t \in \dd u, \; t < T_{b,U}^{-,\gamma} (n)\wedge T_{a,U}^{+,\lambda}(n) \wedge s \bigr)\\
			&	 = \PP_x\bigl( U^{(n)}_t \in \dd u, \; t < T_{b,U}^{-,\gamma}(n) \wedge T_{a,U}^{+,\lambda}(n) \bigr) \1_{\{s < t\}}, 
		\end{align*}
		and similarly for the second integral term of the last inequality. 
		Now, since strong approximation yields that the last integral term in Eq.~\eqref{eq:Thm1-TriangleIneq2} converges to zero as $n \rightarrow \infty$, we can use Fatou's lemma and  Eq.~\eqref{eq:Thm1-TriangleIneq2} to get
		\begin{align}
			\limsup\limits_{n\rightarrow \infty} & \Bigl| \int^\infty_0 \ee^{-qt} \PP_x\bigl( U^{(n)}_t \in \dd u, \; t < T_{b,U}^{-,\gamma}(n) \wedge T_{a,U}^{+,\lambda}(n) \bigr) \dd t - \int^\infty_0 \ee^{-qt} \PP_x\bigl( U_t \in \dd u, \; t < T_{b,U}^{-,\gamma} \wedge T_{a,U}^{+,\lambda} \bigr) \dd t \Bigr| \notag \\
			&\leq  \int^\infty_s \ee^{-qt} \limsup\limits_{n\rightarrow \infty} \PP_x\bigl( U^{(n)}_t \in \dd u, \; t < T_{b,U}^{-,\gamma}(n) \wedge T_{a,U}^{+,\lambda}(n) \bigr) \dd t\notag \\
			&\quad  + \int^\infty_s \ee^{-qt} \PP_x\bigl( U_t \in \dd u, \; t < T_{b,U}^{-,\gamma} \wedge T_{a,U}^{+,\lambda} \bigr) \dd t, \notag
		\end{align}
			and thus, choosing $s$ large enough, we see
		\begin{equation}
			\int^\infty_0 \ee^{-qt} \PP_x\bigl( U^{(n)}_t \in \dd u, \; t < T_{b,U}^{-,\gamma} (n)\wedge T_{a,n}^{+,\lambda}(n) \bigr) \dd t  \rightarrow \int^\infty_0 \ee^{-qt} \PP_x\bigl( U_t \in \dd u, \; t < T_{b,U}^{-,\gamma} \wedge T_{a,U}^{+,\lambda} \bigr) \dd t. \notag
		\end{equation}
			It remains to show that the scale functions of the r.h.s.~of Eq.~\eqref{eq:ReflectedRefracted-DoubleKilledPotential} in the bounded variation case converges to its unbounded variation equivalent, which is shown in Section 3.2 of \cite{F2014}.
		\end{proof}
		\begin{theorem}\label{Thm:SecondMain-ReflectedRefractedDoubleKilledPoissonPotential}
			Let $q \geq 0$. Then, for  $\lambda,\gamma > 0$, $x,b \in [0,a]$ and $u \in [0,\infty)$,
			\begin{align}
				&\quad \mathbb{E}_x \Bigl(\int_0^{\infty} e^{-q t} 1_{\left\{U_t \in \dd u, t < T_{b,U}^{+,\gamma} \wedge T_{a,U}^{-,\lambda} \right\}} \mathrm{d} t\Bigr) \big/ \dd u\notag \\
				&=  \frac{\overline{z}_b^{(q+\lambda,\gamma-\lambda)}(x)  + \lambda \int^x_b \W^{(q+\gamma+\lambda)}(x-y)  \overline{z}_b^{(q+\lambda,\gamma-\lambda)}(y) \dd y }{\ee^{\varphi_{q+\gamma} a} \alpha_2^{(q+\gamma,q+\lambda)}(b) +  \lambda \int^a_b \Z^{(q+\gamma+\lambda)}(a-y,\varphi_{q+\gamma})\overline{z}_b^{(q+\lambda,\gamma-\lambda)}(y) \dd y}  \notag \\
				&\quad \times \Bigl( \Bigl[ \ee^{\varphi_{q+\gamma} a} \alpha_1^{(q+\gamma,q+\lambda)}(b;u) +  \lambda \int^a_b \Z^{(q+\gamma+\lambda)}(a-y,\varphi_{q+\gamma})\overline{w}_b^{(q+\lambda,\gamma-\lambda)}(y;u) \dd y  \Bigr] \1_{\{u \in [0,b)\}} \notag \\
				& \quad \quad + \Z^{(q+\gamma+\lambda)}(a-u,\varphi_{q+\gamma})\1_{\{u \in [b,\infty)\}} \Bigr) -  \Bigl(
				\overline{w}_b^{(q+\lambda,\gamma-\lambda)}(x;u) + \lambda \int^x_b \W^{(q+\gamma+\lambda)}(x-y)  \overline{w}_b^{(q+\lambda,\gamma-\lambda)}(y;u) \dd y \Bigr), \label{eq:ReflectedRefracted-DoubleKilledPotential3}
			\end{align}
			where $\overline{w}_b^{(\cdot,\cdot)}$, $\overline{z}_b^{(\cdot,\cdot)}$, $\alpha_1^{(\cdot,\cdot)}$ and $\alpha_2^{(\cdot,\cdot)}$ are defined in Eqs.~\eqref{eq:SecondGenFunc-RefractedWTilde1},  \eqref{eq:SecondGenFunc-RefractedZTilde1}, \eqref{eq:LimitOfWBar} and \eqref{eq:LimitOfZBar}, respectively. 
		\end{theorem}
		\begin{remark}\upshape{
			The identity $\mathbb{E}_x \bigl(\int_0^{\infty} e^{-q t} 1_{\left\{U_t \in \dd u, t < T_{b,U}^{-,\gamma} \wedge T_{a,U}^{-,\lambda} \right\}} \mathrm{d} t\bigr) \big/ \dd u$ can be proven by using similar arguments as in the proof below  with the necessary adjustments. The form of this identity can be obtained by replacing all instances of $q$ and $\gamma$ in the r.h.s.~of Eq.~\eqref{eq:ReflectedRefracted-DoubleKilledPotential3} with $q+\gamma$ and $-\gamma$, respectively (see Remark \ref{Rem:AboutHowFormsAreTheSame} for further details).}
		\end{remark}	
		
		\begin{proof}
			We consider first the bounded variation case. Let  $e_\eta \sim \text{Exp}(\eta)$, for $\eta = \lambda, \gamma$, be independent of all other random variables and define  $G^{(q)}_{\lambda,\gamma}(x,\dd u) :=\; \E_x\bigl( \int^\infty_0 e^{-q t} 1_{\{U_t \in \dd u, t < T_{b,U}^{+,\gamma} \wedge T_{a,U}^{-,\lambda} \}} \dd t \bigr)$. Recall that $\{U_t, t < \tau_{b,U}^+\} \overset{\dd}{=} \{\wt{X}_t, t < 
			\tau_{b,\wt{X}}^+\}$ w.r.t. $\PP_x$ for $x \in [0,b)$ and, moreover, observe that $T_{b,U}^{+,\gamma} \wedge T_{a,U}^{-,\lambda} \wedge \tau_{b,U}^+ =_\dd e_\lambda \wedge \tau_{b,U}^+$. Hence, by conditioning on $\tau_{b,U}^+$ and using the strong Markov property, it follows that 
			\begin{align}
				G^{(q)}_{\lambda,\gamma}(x,\dd u)
				=& \; \E_x\left( \int^\infty_0 \ee^{-q t} \1_{\{U_t \in \dd u, \; t < e_{\lambda} \wedge \tau_{b,U}^{+} \}} \dd t \right) + \E_x\left( \ee^{-q \tau_{b,U}^{+}} \1_{\{\tau_{b,U}^{+} < e_\lambda \}} \right) G^{(q)}_{\lambda,\gamma}(b, \dd u) \notag \\
				=& \; \E_x\left( \int^\infty_0 \ee^{-(q+\lambda) t} \1_{\{\wt{X}_t \in \dd u, \; t < \tau_{b,\wt{X}}^{+} \}} \dd t \right) + \E_x\left( \ee^{-(q+\lambda) \tau_{b,\wt{X}}^{+}} \1_{\{\tau_{b,\wt{X}}^{+} < \infty \}} \right) G^{(q)}_{\lambda,\gamma}(b, \dd u)  \notag \\
				=& \; \frac{Z^{(q+\lambda)}(x)}{Z^{(q+\lambda)}(b)} \Bigl( {W}^{(q+\lambda)}(b-u) \1_{\{u \in [0,b)\}} \dd u + G^{(q)}_{\lambda,\gamma}(b, \dd u) \Bigr) - {W}^{(q+\lambda)}(x-u) \1_{\{u \in [0,b)\}} \dd u, \label{eq:Thm2-SMPx<bIdentity1}
			\end{align}
			where the last equality follows by using the identities in Eq.~\eqref{eq:ReflectedBelowExitfromAbove}.
			
			Considering $x \in [b,a]$, we notice that $T_{b,U}^{+,\gamma} \wedge T_{a,U}^{-,\lambda} \wedge \tau_{b,U}^- \overset{\dd}{=} e_\gamma \wedge T_{a,U}^{-,\lambda} \wedge \tau_{b,U}^-$ and also that $\{U_t, t < \tau_{b,U}^-\} \overset{\dd}{=} \{Y_t, t < \tau_{b,Y}^-\}$ w.r.t.~$\PP_x$ for these $x$-values. Hence, conditioning on $\tau_{b,Y}^-$ and using the strong Markov property, 
			\begin{align}
				G^{(q)}_{\lambda,\gamma}(x,\dd u) 
				=& \; \E_x\Bigl( \int^\infty_0 e^{-q t} \1_{\{U_t \in \dd u, \; t < T_{a,U}^{-,\lambda} \wedge \tau_{b,u}^- \wedge e_\gamma  \}} \dd t \Bigr) + \E_x\Bigl( e^{-q \tau_{b,U}^-} \1_{ \{ \tau_{b,U}^- < T_{a,U}^{-,\lambda} \wedge e_\gamma \}} G^{(q)}_{\lambda,\gamma}(U_{\tau_{b,U}^-},\dd u) \Bigr) \notag \\
				=& \; \E_x\left( \int^\infty_0 e^{-(q+\gamma) t} \1_{\{Y_t \in \dd u, \; t <  T_{a,Y}^{-,\lambda} \wedge \tau_{b,Y}^- \}} \dd t \right) + \E_x\Bigl( e^{-(q+\gamma) \tau_{b,Y}^-} \1_{ \{ \tau_{b,Y}^- < T_{a,Y}^{-,\lambda} \}} G^{(q)}_{\lambda,\gamma}(Y_{\tau_{b,Y}^-},\dd u) \Bigr) \notag \\
				=& \; \Bigl( \frac{\W^{(q+\gamma+\lambda)}(x-b)}{\Z^{(q+\gamma+\lambda)}(a-b, \varphi_{q+\gamma})} \Z^{(q+\gamma+\lambda)}(a-u, \varphi_{q+\gamma}) - \W^{(q+\gamma+\lambda)}(x-u)\Bigr) \1_{\{u \in [b,\infty)\}}  \dd u \notag \\
				&+ \frac{1}{Z^{(q+\lambda)}(b)} \Bigl( {W}^{(q+\lambda)}(b-u) \1_{\{u \in [0,b)\}} \dd u + G^{(q)}_{\lambda,\gamma}(b, \dd u) \Bigr) \; \E_x\Bigl( e^{-(q+\gamma) \tau_{b,Y}^-} \1_{ \{ \tau_{b,Y}^- < T_{a,Y}^{-,\lambda} \}} Z^{(q+\lambda)}(Y_{\tau_{b,Y}^-}) \Bigr) \notag \\
				&- \E_x\Bigl( e^{-(q+\gamma) \tau_{b,Y}^-} \1_{ \{ \tau_{b,Y}^- < T_{a,Y}^{-,\lambda} \}} W^{(q+\lambda)}(Y_{\tau_{b,Y}^-} - u) \Bigr) \1_{\{u \in [0,b)\}} \dd u , \notag
			\end{align}
			where, for the last equality, the first term follows by using Eq.~\eqref{eq:Landriault-Identity3} and the remaining terms follow by a substitution of Eq.~\eqref{eq:Thm2-SMPx<bIdentity1}.
			
			We observe for the equation above that the first and second expectation terms can be evaluated by using Lemma \ref{lem:SimilartoRonnieLemma} (ii) which yields
			\begin{align}
				&G^{(q)}_{\lambda,\gamma}(x,\dd u) = \;  \Bigl( \frac{\W^{(q+\gamma+\lambda)}(x-b)}{\Z^{(q+\gamma+\lambda)}(a-b, \varphi_{q+\gamma})} \Z^{(q+\gamma+\lambda)}(a-u, \varphi_{q+\gamma}) - \W^{(q+\gamma+\lambda)}(x-u)\Bigr) \1_{\{u \in [b,\infty)\}}  \dd u \notag \\
				& \eqsp +  \frac{1}{Z^{(q+\lambda)}(b)} \Bigl( {W}^{(q+\lambda)}(b-u) \1_{\{u \in [0,b)\}} \dd u + G^{(q)}_{\lambda,\gamma}(b, \dd u) \Bigr) \notag \\
				& \eqsp \times \Bigl[ \overline{z}_b^{(q+\lambda,\gamma-\lambda)}(x) + \lambda \int^x_b \W^{(q+\gamma+\lambda)}(x-y)\overline{z}_b^{(q+\lambda,\gamma-\lambda)}(y) \dd y \notag \\
				&\quad \quad  - \frac{\W^{(q+\gamma+\lambda)}(x-b)}{\Z^{(q+\gamma+\lambda)}(a-b,\varphi_{q+\gamma})}\Bigl( \ee^{\varphi_{q+\gamma} a} \alpha_2^{(q+\gamma,q+\lambda)}(b)  + \lambda \int^a_b \Z^{(q+\gamma+\lambda)}(a-y,\varphi_{q+\gamma}) \overline{z}_b^{(q+\lambda,\gamma-\lambda)}(y) \dd y \Bigr) \Bigr] \notag \\
				& \eqsp - \Bigl[ \overline{w}_b^{(q+\lambda,\gamma-\lambda)}(x;u) + \lambda \int^x_b \W^{(q+\gamma+\lambda)}(x-y)\overline{w}_b^{(q+\lambda,\gamma-\lambda)}(y;u) \dd y \notag \\
				&\quad \quad  - \frac{\W^{(q+\gamma+\lambda)}(x-b)}{\Z^{(q+\gamma+\lambda)}(a-b,\varphi_{q+\gamma})}\Bigl( \ee^{\varphi_{q+\gamma} a} \alpha_1^{(q+\gamma,q+\lambda)}(b;u)  + \lambda \int^a_b \Z^{(q+\gamma+\lambda)}(a-y,\varphi_{q+\gamma}) \overline{w}_b^{(q+\lambda,\gamma-\lambda)}(y;u) \dd y \Bigr) \Bigr] \1_{\{u \in [0,b)\}} \dd u. 
			\end{align}
			Now, we simplify the above equation by noticing for $u \geq b$ that
			\begin{align*}
				\overline{w}_b^{(q+\lambda,\gamma-\lambda)}(x;u) + &\lambda \int^x_b \W^{(q+\gamma+\lambda)}(x-y)\overline{w}_b^{(q+\lambda,\gamma-\lambda)}(y;u) \dd y \notag \\
				&= 	\W^{(q+\gamma)}(x-u) + \lambda \int^x_b \W^{(q+\gamma+\lambda)}(x-y)\W^{(q+\gamma)}(y-u) \dd y \\
				&= \W^{(q+\gamma+\lambda)}(x-u),
			\end{align*}
			which follows by using Eqs.~\eqref{eq:LoeffenLTIdentity1} and \eqref{eq:SecondGenFunc-RefractedWTilde2}, and hence coincides with the $\W^{(q+\gamma+\lambda)}(x-u) \1_{\{u \in [b,\infty)\}}$ term in the above. 
			Thus, using this and grouping the remaining quantities, we finally arrive at the equation
			\begin{align}
				&G^{(q)}_{\lambda,\gamma}(x,\dd u) = \frac{1}{Z^{(q+\lambda)}(b)} \Bigl( {W}^{(q+\lambda)}(b-u) \1_{\{u \in [0,b)\}} \dd u + G^{(q)}_{\lambda,\gamma}(b, \dd u) \Bigr) \notag \\
				& \eqsp \times \Bigl[ \overline{z}_b^{(q+\lambda,\gamma-\lambda)}(x) + \lambda \int^x_b \W^{(q+\gamma+\lambda)}(x-y)\overline{z}_b^{(q+\lambda,\gamma-\lambda)}(y) \dd y \notag \\
				&\quad \quad  - \frac{\W^{(q+\gamma+\lambda)}(x-b)}{\Z^{(q+\gamma+\lambda)}(a-b,\varphi_{q+\gamma})}\Bigl( \ee^{\varphi_{q+\gamma} a} \alpha_2^{(q+\gamma,q+\lambda)}(b)  + \lambda \int^a_b \Z^{(q+\gamma+\lambda)}(a-y,\varphi_{q+\gamma}) \overline{z}_b^{(q+\lambda,\gamma-\lambda)}(y) \dd y \Bigr) \Bigr] \notag \\
				& \eqsp - \Bigl( \overline{w}_b^{(q+\lambda,\gamma-\lambda)}(x;u) + \lambda \int^x_b \W^{(q+\gamma+\lambda)}(x-y)\overline{w}_b^{(q+\lambda,\gamma-\lambda)}(y;u) \dd y \Bigr)  \notag \\
				&\quad + \frac{\W^{(q+\gamma+\lambda)}(x-b)}{\Z^{(q+\gamma+\lambda)}(a-b,\varphi_{q+\gamma})}\Bigl(\Bigl[ \ee^{\varphi_{q+\gamma} a} \alpha_1^{(q+\gamma,q+\lambda)}(b;u)  + \lambda \int^a_b \Z^{(q+\gamma+\lambda)}(a-y,\varphi_{q+\gamma}) \overline{w}_b^{(q+\lambda,\gamma-\lambda)}(y;u) \dd y \Bigr]\1_{\{u \in [0,b)\}} \notag \\
				&\quad \quad+ \Z^{(q+\gamma+\lambda)}(a-u, \varphi_{q+\gamma}) \1_{\{u \in [b,\infty) \}}\Bigr)  \dd u, \label{Thm2-eq:AlmostLastRenewalEq} 
			\end{align}
			and hence what remains is to determine $G^{(q)}_{\lambda,\gamma}(b,\dd u)$. To do this, we notice as in the proof of Theorem \ref{Thm:Main-ReflectedRefractedDoubleKilledPoissonPotential} that  $\overline{w}_b^{(q+\gamma,-\gamma)}(b;u) \dd u = W^{(q+\gamma)}(b-u)\1_{\{u \in [0,b)\}} \dd u$ and $\overline{z}_b^{(q+\gamma,-\gamma)}(b) = Z^{(q+\gamma)}(b)$. Since $\W^{(q)}(0) \neq 0$ in the bounded variation case, we let $x=b$ in Eq.~\eqref{Thm-eq:AlmostLastRenewalEq2}, use the two previously mentioned expressions for $\overline{w}_b^{(q+\gamma,-\gamma)}(b;u)$ and $\overline{z}_b^{(q+\gamma,-\gamma)}(b)$, and solve w.r.t.~$G^{(q)}_{\lambda,\gamma}(b,\dd u)$ to obtain
			\begin{equation}
				\begin{aligned}
					&G^{(q)}_{\lambda,\gamma}(b,\dd u) \big/ \dd u = \frac{Z_b^{(q+\lambda)}(b)}{\ee^{\varphi_{q+\gamma} a} \alpha_2^{(q+\gamma,q+\lambda)}(b) +  \lambda \int^a_b \Z^{(q+\gamma+\lambda)}(a-y,\varphi_{q+\gamma})\overline{z}_b^{(q+\lambda,\gamma-\lambda)}(y) \dd y}  \notag \\
					&\quad \quad \times \Bigl( \Bigl[ \ee^{\varphi_{q+\gamma} a} \alpha_1^{(q+\gamma,q+\lambda)}(b;u) +  \lambda \int^a_b \Z^{(q+\gamma+\lambda)}(a-y,\varphi_{q+\gamma})\overline{w}_b^{(q+\lambda,\gamma-\lambda)}(y;u) \dd y  \Bigr] \1_{\{u \in [0,b)\}} \notag \\
					& \quad \quad + \Z^{(q+\gamma+\lambda)}(a-u,\varphi_{q+\gamma})\1_{\{u \in [b,\infty)\}} \Bigr) -  
					W^{(q+\lambda)}(x-u) \1_{\{u \in [0,b)\}}.
				\end{aligned} \notag
			\end{equation}
			The results for $x \in [0,b)$ and $x \in [b,a]$ then follow by substituting the above equation into Eqs.~\eqref{eq:Thm2-SMPx<bIdentity1} and \eqref{Thm2-eq:AlmostLastRenewalEq}, respectively.
			
			To prove the unbounded variation case, we use strong approximation in a similar way as in the proof of Theorem \ref{Thm:Main-ReflectedRefractedDoubleKilledPoissonPotential}. 
		\end{proof}

		\section{Poissonian potential measures with one-sided killing} \label{Sec:LimitsofMainTheorems}
		In this section, we use Theorem \ref{Thm:Main-ReflectedRefractedDoubleKilledPoissonPotential} and \ref{Thm:SecondMain-ReflectedRefractedDoubleKilledPoissonPotential} along with further limiting identities of the scale functions to derive one-sided Poissonian potential measures for the refracted-reflected  process $U$. 	
		To derive these identities, we shall need to prove that $T_{a,X}^{+,\lambda} \overset{\mathrm{p}}{\rightarrow} \tau_{a,X}^{+}$, where $\overset{\mathrm{p}}{\rightarrow}$ means convergence in probability (and similarly $T_{b,X}^{-,\lambda}\overset{\mathrm{p}}{\rightarrow} \tau_{b,X}^{-}$) as the observation rate $\lambda \rightarrow \infty$. This is the purpose of the next lemma. 
		
		\begin{lemma}\label{lem:ConvergenceInProbofPoisTimes}
			Let $a,b \in \R$ and let $T_{a,X}^{+,\lambda}$ and $T_{b,X}^{-,\lambda}$ be defined as previously. Then, $T_{a,X}^{+,\lambda} \overset{\mathrm{p}}{\rightarrow} \tau_{a,X}^{+}$ and  $T_{b,X}^{-,\lambda} \overset{\mathrm{p}}{\rightarrow} \tau_{b,X}^{-}$ as $\lambda \rightarrow \infty$. 
		\end{lemma}
		
		\begin{proof}
			We shall condition on the event $\{\tau_{a,X}^+ < \infty\}$ throughout the proof since we have that $\tau_{a,X}^+ \leq T_{a,X}^{+,\lambda}$. Now, for all $\varepsilon > 0$, we have that there exists at least one pair $s,\delta > 0$ s.t.~$X_t > a$ for all $t \in [s, s+\delta) \subset (\tau_{a,X}^+, \tau_{a,X}^+ + \varepsilon)$, $\PP_x$-a.s.~Indeed, this follows directly from the definition of $\tau_{a,X}^+$ and the right-continuity of $X$. Clearly  $s$ and $\delta$ are random, determined by the behaviour of $X$ and independent of the Poisson process $N$. 
			
			Now, define $\mathcal{E}_a := \{t > 0 : X_t > a\}$. Then $[s,s+\delta) \subset \mathcal{E}_a$, and so $\mathcal{E}_a$ is a non-empty countable union of disjoint intervals (this follows by the right-continuity of $X$ and the density of the rational numbers in the real line). Hence, we have for all $\varepsilon > 0$ that
			\begin{align*}
				\PP_x(T_{a,X}^{+,\lambda} - \tau_{a,X}^+ > \varepsilon) &= \PP_x(\,\{\text{There are no Poisson arrivals in } \mathcal{E}_a \cap ( \tau_{a,X}^+,  \tau_{a,X}^+ + \varepsilon)\}\,) \\
				&\leq \PP_x( \, \{\text{There are no Poisson arrivals in } [s,s+\delta) \} \, ) \\
				&= \PP_x( N(s+\delta) - N(s) = 0 ) = \E_{x} \bigl( \PP_x( N(s+\delta) - N(s) = 0 \;\vert \; s,\delta \; ) \bigr) \\
				&= \E_{x} \bigl( \PP_x( N(\delta) = 0 \;\vert \; \delta \; ) \bigr) 
				= \int^\infty_0 \ee^{-\lambda t} \, \PP_x(\delta \in \dd t),
			\end{align*}	
			and thus taking $\lambda \rightarrow \infty$ on both sides of the inequality and using dominated convergence yields the assertion. The proof for $T_{b,X}^{-,\lambda} \overset{\mathrm{p}}{\rightarrow} \tau_{b,X}^{-}$ as $\lambda \rightarrow \infty$ follows similarly. 
		\end{proof}
		\noindent We may now derive the limit identities of the previously derived Poissonian potential measures. 
		
		\color{black}
		\begin{proposition}\label{Thm:MainPoissonPotentialAbove} The following results hold.
			\begin{itemize}
				\item[\upshape{(i)}] 	Let $q \geq 0$. Then, for $\lambda > 0$, $x,b \in [0,a]$ and $u \in [0,\infty)$,
				\begin{equation}
					\begin{aligned}
						&\mathbb{E}_x\Bigl(\int_0^{\infty} \ee^{-q t} 1_{\left\{U_t \in \dd u, \; t < T_{a,U}^{+,\lambda} \right\}} \mathrm{d} t\Bigr) \big/ \dd u \\
						&=  \frac{\lambda \int^\infty_0 \ee^{-\varphi_{q+\lambda}y} w^{(q)}(y+a;u) \dd y \1_{\{u \in [0,a]\}} + \ee^{-\varphi_{q+\lambda}(u-a)}\1_{\{u \in (a,\infty)\}}}{\lambda \int^\infty_0 \ee^{-\varphi_{q+\lambda}y} z^{(q)}(y+a) \dd y} z^{(q)}(x) -  w^{(q)}(x;u) .
					\end{aligned}  \label{eq:ReflectedRefracted-KilledPotential}
				\end{equation}
				\item[\upshape{(ii)}] Let $q \geq 0$. Then, for $\lambda > 0$, $x,b \in [0,a]$ and $u \in [0,\infty)$,
				\begin{align}
					&\quad \mathbb{E}_x \Bigl(\int_0^{\infty} e^{-q t} 1_{\left\{U_t \in \dd u, t <  T_{a,U}^{-,\lambda} \right\}} \mathrm{d} t\Bigr) \big/ \dd u\notag \\
					&=  \frac{\overline{z}_b^{(q+\lambda,-\lambda)}(x)  + \lambda \int^x_b \W^{(q+\lambda)}(x-y)  \overline{z}_b^{(q+\lambda,-\lambda)}(y) \dd y }{\ee^{\varphi_{q} a} \alpha_2^{(q,q+\lambda)}(b) +  \lambda \int^a_b \Z^{(q+\lambda)}(a-y,\varphi_{q})\overline{z}_b^{(q+\lambda,-\lambda)}(y) \dd y}  \notag \\
					&\quad \times \Bigl( \Bigl[ \ee^{\varphi_{q} a} \alpha_1^{(q,q+\lambda)}(b;u) +  \lambda \int^a_b \Z^{(q+\lambda)}(a-y,\varphi_{q})\overline{w}_b^{(q+\lambda,-\lambda)}(y;u) \dd y  \Bigr] \1_{\{u \in [0,b)\}} \notag \\
					& \quad \quad + \Z^{(q+\lambda)}(a-u,\varphi_{q})\1_{\{u \in [b,\infty)\}} \Bigr) -  \Bigl(
					\overline{w}_b^{(q+\lambda,-\lambda)}(x;u) + \lambda \int^x_b \W^{(q+\lambda)}(x-y)  \overline{w}_b^{(q+\lambda,-\lambda)}(y;u) \dd y \Bigr). \label{eq:ReflectedRefracted-KilledPotentialforTa^-}
				\end{align}
				\item[\upshape{(iii)}] For $q, \gamma \geq 0$ and $x,b,u \in [0,a]$,
				\begin{equation}
					\mathbb{E}_x\Bigl(\int_0^{\infty} \ee^{-q t} 1_{\left\{U_t \in \dd u, t < T_{b,U}^{-,\gamma} \wedge \tau_{a,U}^+ \right\}} \mathrm{d} t\Bigr) \big/ \dd u =   \frac{\ovl{z}^{(q+\gamma,-\gamma)}_b(x)}{\ovl{z}^{(q+\gamma,-\gamma)}_b(a)} \ovl{w}_b^{(q+\gamma,-\gamma)}(a;u) -  \ovl{w}_b^{(q+\gamma,-\gamma)}(x;u). \label{eq:ReflectedRefracted-KilledPotential1}
				\end{equation}
				
				\item[\upshape{(iv)}] For $q, \gamma \geq 0$ and $x,b,u \in [0,a]$,	\begin{equation}
					\mathbb{E}_x\Bigl(\int_0^{\infty} \ee^{-q t} 1_{\left\{U_t \in \dd u, t < T_{b,U}^{+,\gamma} \wedge \tau_{a,U}^+ \right\}} \mathrm{d} t\Bigr) \big/ \dd u =   \frac{\overline{z}^{(q,\gamma)}_b(x)}{\overline{z}^{(q,\gamma)}_b(a)} \overline{w}_b^{(q,\gamma)}(a;u) -  \overline{w}_b^{(q,\gamma)}(x;u). \label{eq:ReflectedRefracted-KilledPotential2}
				\end{equation}
				
				\item[\upshape{(v)}] For  $q, \gamma \geq 0$, $x,b \in [0,a]$, and $u \geq 0$,
				\begin{equation}
					\mathbb{E}_x\Bigl(\int_0^{\infty} \ee^{-q t} 1_{\left\{U_t \in \dd u, t < T_{b,U}^{-,\gamma} \right\}} \mathrm{d} t\Bigr) \big/ \dd u = {\ovl{z}^{(q+\gamma,-\gamma)}_b(x)} \frac{\alpha_1^{(q,q+\gamma)}(b;u)}{\alpha_1^{(q,q+\gamma)}(b)} - \ovl{w}_b^{(q+\gamma,-\gamma)}(x;u), \; \notag 
				\end{equation}
				where $\alpha_1^{(\cdot,\cdot)}$ and $\alpha_2^{(\cdot,\cdot)}$ are defined in Eqs.~\eqref{eq:LimitOfWBar} and \eqref{eq:LimitOfZBar}, respectively. 
				\item[\upshape{(vi)}] For $q,\gamma \geq 0$, $x,b \in [0,a]$, and $u \geq 0$,
				\begin{equation}
					\mathbb{E}_x\Bigl(\int_0^{\infty} \ee^{-q t} 1_{\left\{U_t \in \dd u, t < T_{b,U}^{+,\gamma} \right\}} \mathrm{d} t\Bigr) \big/ \dd u =  {\ovl{z}^{(q,\gamma)}_b(x)}
					\frac{\gamma \int^\infty_0 \ee^{- \varphi_{q+\gamma} y}  w^{(q)}(y+b;u) \dd y}{\gamma \int^\infty_0 \ee^{- \varphi_{q+\gamma} y}  z^{(q)}(y+b) \dd y} - \ovl{w}_b^{(q,\gamma)}(x;u). \; \notag 
				\end{equation}
			\end{itemize}
		\end{proposition}
		\begin{proof}
			In all parts of the proof below dominated convergence theorem is used implicitly.
			
			\noindent  	(i)	Observe that  $\gamma \downarrow 0$ implies $T_{b,U}^{-,\gamma} \uparrow \infty$,  and also that (from Eqs.~\eqref{eq:SecondGenFunc-RefractedWTilde1} and \eqref{eq:SecondGenFunc-RefractedZTilde1})  
			\begin{equation}
				\lim\limits_{\gamma \downarrow 0} \overline{w}_b^{(q+\gamma,-\gamma)}(x;u) = w^{(q)}(x;u) \quad \text{ and } \quad  \lim\limits_{\gamma \downarrow 0} \overline{z}_b^{(q+\gamma,-\gamma)}(x) = z^{(q)}(x). \notag
			\end{equation}
			Then, taking $\gamma \downarrow 0$ in Eq.~\eqref{eq:ReflectedRefracted-DoubleKilledPotential}, the result follows. 
			\\[3pt]
			\noindent \upshape{(ii)} Using a similar reasoning as for the proof of (i), we take $\gamma \downarrow 0$ in Eq.~\eqref{eq:ReflectedRefracted-DoubleKilledPotential3} and the result follows. 
			\\[3pt]
			\noindent \upshape{(iii)} First observe that $\varphi_{q+\lambda}\ee^{-\varphi_{q+\lambda}(u-a)}\1_{\{u \in (a,\infty)\}} \rightarrow 0$ as $\lambda \rightarrow \infty$. Then, dividing the existing $\lambda$ terms, 
			multiplying and dividing by $\varphi_{q+\lambda}$ in both the numerator and denominator in the fraction terms of Eqs.~\eqref{eq:ReflectedRefracted-DoubleKilledPotential} and then using the initial value theorem of Laplace transforms (see Remark \ref{rem:InitialValueTheorem}) yields the result. 
			\\[3pt]
			\noindent \upshape{(iv)} By using the identity that is obtained by changing the parameters $q$ and $\gamma$ appropriately (see Remark \ref{Rem:AboutHowFormsAreTheSame} for further details), a similar method as in \upshape{(iii)} can be used to derive the result. 
			\\[3pt]
			\noindent \upshape{(v)} Taking $a\rightarrow \infty$ in part  \upshape{(iii)} and using Eq.~\eqref{eq:LimitOfWandZBar} we obtain the result. 
			\\[3pt]
			\noindent \upshape{(vi)} 
			By a slight adaptation of the result and proof of Lemma 9 of \cite{BBR2009}, it can easily be shown, since $\varphi_{q+\gamma} > \varphi_{q} \geq \Phi_q$, that 
			\[
			\lim\limits_{a \rightarrow \infty} \frac{W^{(q)}(a - u)}{\W^{(q+\gamma)}(a)} = \lim\limits_{a \rightarrow \infty} \frac{\W^{(q)}(a - u)}{\W^{(q+\gamma)}(a)} = 0, \quad \quad u \geq 0,  
			\]
			and thus, from Eqs.~\eqref{eq:RefractedWScaleFunc} and \eqref{eq:RefractedZScaleFunc}, that 
			\begin{equation}
				\lim\limits_{a \rightarrow \infty } \frac{w^{(q)}(a;u)}{\W^{(q+\gamma)}(a)}= \lim\limits_{a \rightarrow \infty } \frac{z^{(q)}(a)}{\W^{(q+\gamma)}(a)} = 0, \quad \quad u \geq 0.   \label{eq:TempLimitforProp(v)}
			\end{equation}
			\color{black}
			Then, by using  Eqs.~\eqref{eq:SecondGenFunc-RefractedWTilde1} and \eqref{eq:SecondGenFunc-RefractedZTilde1} along with Eqs.~\eqref{eq:LimitofRatioofScaleFuncs} and \eqref{eq:TempLimitforProp(v)}, we get that
			\begin{align}
				\lim\limits_{a \rightarrow \infty} \frac{\ovl{w}^{(q,\gamma)}_b(a;u)}{\W^{(q+\gamma)}(a)} 
				&= \ee^{- \varphi_{q+\gamma} b} \times \gamma \int^\infty_0 \ee^{- \varphi_{q+\gamma} y}  w^{(q)}(y+b;u) \dd y,  \notag
			\end{align}
			and 
			\begin{align}
				\lim\limits_{a \rightarrow \infty} \frac{\ovl{z}^{(q,\gamma)}_b(a)}{\W^{(q+\gamma)}(a)} 
				&= \ee^{- \varphi_{q+\gamma} b} \times \gamma \int^\infty_0 \ee^{- \varphi_{q+\gamma} y}  z^{(q)}(y+b) \dd y,  \notag
			\end{align}
			respectively. We complete the proof by substituting the above limits in part   \upshape{(iv)} as $a\rightarrow \infty$. 
		\end{proof}
		
		\begin{remark}\label{cor:ReflectedRefractedPotentialAbove}\upshape{
				
				For $x,u \in [0,a]$, observe that
				\begin{equation}
					\varphi_{q+\lambda} \int^\infty_0 \ee^{-\varphi_{q+\lambda}y} w^{(q)}(y+x;u) \dd y \rightarrow w^{(q)}(x;u), \quad \text{ and } \quad \varphi_{q+\lambda} \int^\infty_0 \ee^{-\varphi_{q+\lambda}y} z^{(q)}(y+x) \dd y \rightarrow z^{(q)}(x), \notag
				\end{equation}
				as $\lambda \rightarrow \infty$, a consequence of the initial value Theorem of Laplace transforms (see Remark \ref{rem:InitialValueTheorem}), and also that $\varphi_{q+\lambda}\ee^{-\varphi_{q+\lambda}(u-a)}\1_{\{u \in (a,\infty)\}} \rightarrow 0$ as $\lambda \rightarrow \infty$. Then, by multiplying and dividing by $\varphi_{q+\lambda}$ and taking $\lambda \rightarrow \infty $ in Proposition \ref{Thm:MainPoissonPotentialAbove} \textnormal{(i)}, we obtain Theorem 4.1 in \cite{PY2018} which can also be obtained by letting $\gamma \rightarrow 0$ and using dominated convergence in Eqs.~\eqref{eq:ReflectedRefracted-KilledPotential1} and \eqref{eq:ReflectedRefracted-KilledPotential2}.
			}
		\end{remark}
		
		\section{Occupation Times}
		\label{sec:occup}
		
		In this section, we are interested in computing the occupation times of the process $U$ above and below the level of refraction $b$ but w.r.t.~the minimum of the Poissonian exit times $T_{a,U}^{+,\lambda} \wedge T_{b,U}^{-,\gamma}$ and $T_{a,U}^{-,\lambda} \wedge T_{b,U}^{+,\gamma}$. Namely, for $a>b>0$, we compute the joint Laplace transforms of 
		\[
		\Bigl(T_{a,U}^{+,\lambda} \wedge T_{b,U}^{-,\gamma},\int_0^{T_{a,U}^{+,\lambda} \wedge T_{b,U}^{-,\gamma}} 1_{\left\{U_s<b\right\}} \mathrm{d} s\Bigr) \quad \text { and } \quad \Bigl(T_{a,U}^{-,\lambda} \wedge T_{b,U}^{+,\gamma},\int_0^{T_{a,U}^{-,\lambda} \wedge T_{b,U}^{+,\gamma}} 1_{\left\{U_s < b\right\}} \mathrm{d} s\Bigr),
		\]
		conditionally on the event $T_{a,U}^{+,\lambda} \wedge T_{b,U}^{-,\gamma}<\infty$.  We begin by first considering the cases for which either $T_{a,U}^{+,\lambda} < T_{b,U}^{-,\gamma}$ or $ T_{b,U}^{-,\gamma} < T_{a,U}^{+,\lambda}$ and use them to derive the required joint Laplace transforms, as  given above. It is also worthwhile noting that the joint Laplace transforms corresponding to the minimums of $T_{a,U}^{+,\lambda} \wedge T_{b,U}^{+,\gamma}$ and $T_{a,U}^{-,\lambda} \wedge T_{b,U}^{-,\gamma}$ can be derived by using similar methods as in the upcoming, and we thus exclude these results. 
		
		For the proceeding results, we shall denote, for $p \geq 0$, $\lambda_p = \lambda + p$ and $\gamma_p = \gamma + p$. Furthermore, in both Propositions \ref{Prop:OccupationTimes1} and \ref{Prop:OccupationTimes2} below,  we shall use strong approximation in the same way as in Theorem \ref{Thm:Main-ReflectedRefractedDoubleKilledPoissonPotential} (see also the proof of Theorem 1 in \cite{LRZ2014}). It is hence sufficient to prove Propositions \ref{Prop:OccupationTimes1} and \ref{Prop:OccupationTimes2} for the cases of bounded variation only. 
		\begin{proposition}\label{Prop:OccupationTimes1}
			Let $p,q \geq 0$ and $\lambda,\gamma> 0$. Then, for $x,b \in [0,a]$, we have that 
			\begin{align}
				\E_x \Bigl( \ee^{-q T_{a,U}^{+,\lambda} - p \int_0^{T_{a,U}^{+,\lambda}} 1_{\left\{U_s<b\right\}} \mathrm{d} s} \; \1_{\{T_{a,U}^{+,\lambda} < T_{b,U}^{-,\gamma}\}}\Bigr) &=  \frac{\overline{z}_b^{(q+\gamma_p, -\gamma_p)}(x)}{\varphi_{q+\lambda} \int^\infty_0 \ee^{-\varphi_{q+\lambda}y}\overline{z}_b^{(q+\gamma_p, -\gamma_p)}(y+a) \dd y}, \label{Eq:OccupationProp1-Identity1} 
			\end{align} 
			where $\overline{z}_b^{(\cdot, \cdot)}$ is defined in Eq.~\eqref{eq:SecondGenFunc-RefractedZTilde1}. 
		\end{proposition}
		
		\begin{remark}\upshape{
			A straightforward argument can now be used to get   $\E_x \bigl( \ee^{-q T_{a,U}^{+,\lambda} - p \int_0^{T_{a,U}^{+,\lambda}} 1_{\left\{U_s>b\right\}} \mathrm{d} s} $ $ \1_{\{T_{a,U}^{+,\lambda} < T_{b,U}^{-,\gamma} \}}\bigr)$ from Eq.~\eqref{Eq:OccupationProp1-Identity1}. Indeed, from Lemma 4.1 of \cite{PY2018}, we have that $\int_0^{T_{a,U}^{+,\lambda}} 1_{\left\{U_s=b\right\}} \mathrm{d} s = 0$ a.s.~and therefore that 
			\begin{align}
				\E_x \Bigl( \ee^{-q T_{a,U}^{+,\lambda} - p \int_0^{T_{a,U}^{+,\lambda}} 1_{\left\{U_s>b\right\}} \mathrm{d} s} \; \1_{\{T_{a,U}^{+,\lambda} < T_{b,U}^{-,\gamma}\}}\Bigr) &= \E_x \Bigl( \ee^{-(q+p) T_{a,U}^{+,\lambda} + p \int_0^{T_{a,U}^{+,\lambda}} 1_{\left\{U_s < b\right\}} \mathrm{d} s} \; \1_{\{T_{a,U}^{+,\lambda} < T_{b,U}^{-,\gamma}\}}\Bigr), \notag
			\end{align}
			showing that the above  l.h.s.~expectation  is obtained from Eq.~\eqref{Eq:OccupationProp1-Identity1} by replacing all instances of $q$ and $p$ therein by $q+p$ and $-p$, respectively\color{black}.  To be more precise, we have that instances of $\varphi_{q+\lambda}$ and  $\overline{z}_b^{(q+\gamma_p,-\gamma_p)}$ are to be replaced by $\varphi_{q+\lambda_p}$ and  $\overline{z}_b^{(q+\gamma,p-\gamma)}$, respectively, in the r.h.s.~of Eq.~\eqref{Eq:OccupationProp1-Identity1} to obtain the desired occupation time identity. This reasoning hence implies that the occupation times above the level $b$ (corresponding to indicators of the form $\1_{\{U_s > b\}}$)
			can essentially be obtained from the identities which shall be derived by a parameter change. We hence omit these quantities for brevity.}
		\end{remark}
		\begin{proof}
			Let $h^{(q,p)}_{\lambda,\gamma}(x) :=\; \E_x \bigl( \ee^{-q T_{a,U}^{+,\lambda} - p \int_0^{T_{a,U}^{+,\lambda}} 1_{\left\{U_s<b\right\}} \mathrm{d} s} \; \1_{\{T_{a,U}^{+,\lambda} < T_{b,U}^{-,\gamma}\}}\bigr)$, and recall both that $T_{b,U}^{-,\gamma} \wedge \tau_{b,\wt{X}}^+  \overset{\dd}{=} e_{\gamma} \wedge \tau_{b,\wt{X}}^+$, for $e_\gamma \sim \text{Exp}(\gamma)$ that is independent of all other random variables, and that $ \{U_t, t < \tau_{b,U}^+\} \overset{\dd}{=} \{\wt{X}_t, t < 
			\tau_{b,\wt{X}}^+\}$ w.r.t.~$\PP_x$ for $x \in [0,b)$. Hence, for $x \in [0,b)$, by conditioning on $\tau_{b,U}^+$ and the strong Markov property,  
			\begin{align}
				h^{(q,p)}_{\lambda,\gamma}(x)
				=& \; \E_x \Bigl( \ee^{-(q+p) {\tau}_{b,\wt{X}}^{+} } \; \1_{\{\tau_{b,\wt{X}}^+ < e_{\gamma} \}}\Bigr) 	h^{(q,p)}_{\lambda,\gamma}(b) 
				= \frac{Z^{(q+\gamma_p)}(x)}{Z^{(q+\gamma_p)}(b)}h^{(q,p)}_{\lambda,\gamma}(b) , \label{eq:OccupationProp-SMPx<bIdentity1}
			\end{align}
			where the last equality follows by using Eq.~\eqref{eq:ReflectedBelowExitfromAbove}.
			
			Now, considering $x \in [b,a]$ and noticing that $\{U_t, t < \tau_{b,U}^-\} \overset{\dd}{=} \{Y_t, t < \tau_{b,Y}^-\}$ w.r.t.~$\PP_x$ for these $x$-values, we condition on $\tau_{b,U}^-$ and use the strong Markov property to get
			\begin{align}
				h^{(q,p)}_{\lambda,\gamma}(x)
				=& \; \E_x \Bigl( \ee^{-q T_{a,U}^{+,\lambda} } \1_{\{ T_{a,U}^{+,\lambda} < \tau_{b,U}^- \}}\Bigr) + \E_x \Bigl( \ee^{-q \tau_{b,U}^-}  \1_{\{\tau_{b,U}^- < T_{a,U}^{+,\lambda} \}} h^{(q,p)}_{\lambda,\gamma}(U_{\tau_{b,U}^-}) \Bigr) \notag \\
				=&\; \E_x \Bigl( \ee^{-q T_{a,Y}^{+,\lambda} } \1_{\{ T_{a,Y}^{+,\lambda} < \tau_{b,Y}^- \}}\Bigr) +  \frac{h^{(q,p)}_{\lambda,\gamma}(b)}{Z^{(q+\gamma_p)}(b)}\E_x \Bigl( \ee^{-q \tau_{b,Y}^-}  \1_{\{\tau_{b,Y}^- < T_{a,Y}^{+,\lambda} \}} Z^{(q+\gamma_p)}(Y_{\tau_{b,Y}^-}) \Bigr),  
				\label{eq:OccupationProp-NewIntermediate}
			\end{align}
			where the last equality follows by a substitution of Eq.~\eqref{eq:OccupationProp-SMPx<bIdentity1}.
			
			To determine the first expectation term in the above equation, we use Lemma \ref{lem:PoissonPotentialsAndFluctuations} (ii) to deduce, by a conditioning argument and Fubini's theorem, that
			\begin{align}
				\E_x\bigl( \ee^{-q T_{a,Y}^{+,\lambda}} \1_{\{T_{a,Y}^{+,\lambda} < \tau_{b,Y}^-\}} \bigr) &= \int^\infty_a
				\E_x\bigl( \ee^{-q T_{a,Y}^{+,\lambda}} \1_{\{Y_{T_{a,Y}^{+,\lambda}} \in \dd y, \; T_{a,Y}^{+,\lambda} < \tau_{b,Y}^-\}} \bigr) 
				= \frac{\lambda}{\varphi_{q+\lambda}} \frac{\W^{(q)}(x-b)}{\Z^{(q)}(a-b,\varphi_{q+\lambda})}. \notag
			\end{align}
			Then, by using the above equation and Lemma \ref{lem:SimilartoRonnieLemma}, Eq.~\eqref{eq:OccupationProp-NewIntermediate} becomes
			\begin{align}
				h^{(q,p)}_{\lambda,\gamma}(x)	=&\; \frac{\lambda}{\varphi_{q+\lambda}} \frac{\W^{(q)}(x-b)}{\Z^{(q)}(a-b,\varphi_{q+\lambda})} \notag\\
				&+  \frac{h^{(q,p)}_{\lambda,\gamma}(b)}{Z^{(q+\gamma_p)}(b)} \Bigl( \overline{z}_b^{(q+\gamma_p,-\gamma_p)}(x) - \frac{\W^{(q)}(x-b)}{\Z^{(q)}(a-b,\varphi_{q+\lambda})}\lambda \int^\infty_0 \ee^{-\varphi_{q+\lambda}y}\overline{z}_b^{(q+\gamma_p,-\gamma_p)}(y+a)\dd y \Bigr). \label{eq:OccupationProp-FinalSMP}
			\end{align}
			Now,  let $x=b$ in the above and solve w.r.t.~$h^{(q,p)}_{\lambda,\gamma}(b)$ to find that
			\begin{equation}
				h^{(q,p)}_{\lambda,\gamma}(b) =  \frac{Z^{(q+\gamma_p)}(b)}{\varphi_{q+\lambda} \int^\infty_0 \ee^{-\varphi_{q+\lambda}y}\overline{z}_b^{(q+\gamma_p, -\gamma_p)}(y+a) \dd y}. \notag
			\end{equation}
			Lastly, by noticing that $\overline{z}_b^{(q+\gamma_p,-\gamma_p)}(x) = Z^{(q+\gamma_p)}(x)$ for $x \in [0,b)$, we substitute the above quantity into  Eqs.~\eqref{eq:OccupationProp-SMPx<bIdentity1} and \eqref{eq:OccupationProp-FinalSMP} to obtain the form of Eq.~\eqref{Eq:OccupationProp1-Identity1} for $x \in [0,a]$ which concludes the proof.
		\end{proof}
		
		\begin{remark}\label{Cor:OccupationTimes1}\upshape{
			\upshape{By letting $\gamma \rightarrow 0$, noticing that $\gamma_p \rightarrow p$ and $T_{b,U}^{-,\gamma} \rightarrow \infty$} and using dominated convergence, Eq.~\eqref{Eq:OccupationProp1-Identity1} reduces to
			\begin{align}
				\E_x \Bigl( \ee^{-q T_{a,U}^{+,\lambda} - p \int_0^{T_{a,U}^{+,\lambda}} 1_{\left\{U_s<b\right\}} \mathrm{d} s} \; \1_{\{T_{a,U}^{+,\lambda} < \infty\}}\Bigr) &=  \frac{\overline{z}_b^{(q+p, -p)}(x)}{\varphi_{q+\lambda} \int^\infty_0 \ee^{-\varphi_{q+\lambda}y}\overline{z}_b^{(q+p, -p)}(y+a) \dd y}. \notag 
			\end{align} }
		\end{remark}
		
		\begin{proposition}\label{Prop:OccupationTimes2}
			Let $p,q \geq 0$ and $\lambda, \gamma> 0$. Then, for $x,b \in [0,a]$, we have that
			\begin{align}
				\E_x \Bigl( \ee^{-q T_{b,U}^{-,\gamma} - p \int_0^{T_{b,U}^{-,\gamma}} 1_{\left\{U_s<b\right\}} \mathrm{d} s} \; \1_{\{T_{b,U}^{-,\gamma} < T_{a,U}^{+,\lambda}\}}\Bigr) &=   \frac{\int^\infty_0 \ee^{-\varphi_{q+\lambda}y}\, \Bigl( \gamma \int^b_0 \overline{w}_b^{(q+\gamma_p,-\gamma_p)}(y+a;u) \dd u\Bigr) \dd y}{\int^\infty_0 \ee^{-\varphi_{q+\lambda}y} \,  \overline{z}_b^{(q+\gamma_p,-\gamma_p)}(y+a) \dd y} \overline{z}_b^{(q+\gamma_p,-\gamma_p)}(x) \notag \\
				&\quad - \gamma \int^b_0 \overline{w}_b^{(q+\gamma_p,-\gamma_p)}(x;u) \dd u \label{Eq:OccupationProp2-Identity1}, 
			\end{align} 
			where $\overline{w}_b^{(\cdot, \cdot)}$ and $\overline{z}_b^{(\cdot, \cdot)}$ are defined in Eqs.~\eqref{eq:SecondGenFunc-RefractedWTilde1} and \eqref{eq:SecondGenFunc-RefractedZTilde1}, respectively. 
		\end{proposition}
		
		\begin{proof}
			Let $g^{(q,p)}_{\lambda,\gamma}(x) :=\; \E_x \bigl( \ee^{-q T_{b,U}^{-,\gamma} - p \int_0^{T_{b,U}^{-,\gamma}} 1_{\left\{U_s<b\right\}} \mathrm{d} s} \; \1_{\{T_{b,U}^{-,\gamma} < T_{a,U}^{+,\lambda}\}}\bigr)$, and recall both that $T_{b,U}^{-,\gamma} \wedge \tau_{b,\wt{X}}^+  \overset{\dd}{=} e_{\gamma} \wedge \tau_{b,\wt{X}}^+$, for $e_\gamma \sim \text{Exp}(\gamma)$ that is independent of all other random variables, and that $ \{U_t, t < \tau_{b,U}^+\} \overset{\dd}{=} \{\wt{X}_t, t < 
			\tau_{b,\wt{X}}^+\}$ w.r.t.~$\PP_x$ for $x \in [0,b)$. Then, for $x \in [0,b)$, by conditioning on $\tau_{b,U}^+$, using the  strong Markov property along with Eq.~\eqref{eq:ReflectedBelowExitfromAbove}, we have that 
			\begin{align}
				g^{(q,p)}_{\lambda,\gamma}(x)
				=& \; \E_x \Bigl( \ee^{-(q+p) e_{\gamma}} \; \1_{\{e_{\gamma} < \tau_{b,\wt{X}}^+ \}}\Bigr) + \E_x \Bigl( \ee^{-(q+p) \tau_{b,\wt{X}}^+}\; \1_{\{\tau_{b,\wt{X}}^+ < e_{\gamma}  \}}\Bigr) g^{(q,p)}_{\lambda,\gamma}(b) \notag \\
				&= \gamma \int^b_0 \E_x \Bigl( \int^\infty_0 \ee^{(q+\gamma_p)t} \; \1_{\{\widetilde{X}_t \in \dd u, \, t <  \tau_{b,\wt{X}}^+\}} \dd t\Bigr) +  \frac{Z^{(q+\gamma_p)}(x)}{Z^{(q+\gamma_p)}(b)} g^{(q,p)}_{\lambda,\gamma}(b) \notag \\
				&= \frac{Z^{(q+\gamma_p)}(x)}{Z^{(q+\gamma_p)}(b)} \Bigl( g^{(q,p)}_{\lambda,\gamma}(b)+\gamma \int^b_0 W^{(q+\gamma_p)}(b-u) \dd u \Bigr) - \gamma \int^b_0 W^{(q+\gamma_p)}(x-u) \dd u.  \label{eq:OccupationProp2-SMPx<bIdentity1}
			\end{align}

			For $x \in [b,a]$, notice that $ \{U_t, t < \tau_{b,U}^-\} \overset{\dd}{=} \{Y_t, t < \tau_{b,Y}^-\}$ w.r.t.~$\PP_x$ for these $x$-values, and so, by conditioning on $\tau_{b,U}^-$ and using the strong Markov property, 
			\begin{align}
				g^{(q,p)}_{\lambda,\gamma}(x)
				=& \; \E_x \Bigl( \ee^{-q \tau_{b,U}^-}  \1_{\{\tau_{b,U}^- < T_{a,U}^{+,\lambda} \}} g^{(q,p)}_{\lambda,\gamma}(U_{\tau_{b,U}^-}) \Bigr) \notag \\
				=& \; \frac{1}{Z^{(q+\gamma_p)}(b)} \Bigl( g^{(q,p)}_{\lambda,\gamma}(b)+\gamma \int^b_0 W^{(q+\gamma_p)}(b-u) \dd u \Bigr) \,  \E_x \Bigl( \ee^{-q \tau_{b,Y}^-}  \1_{\{\tau_{b,Y}^- < T_{a,U}^{+,\lambda} \}} Z^{(q+\gamma_p)}(Y_{\tau_{b,Y}^-}) \Bigr) \notag \\
				&\quad - \gamma \int^b_0 \E_x \Bigl( \ee^{-q \tau_{b,Y}^-}  \1_{\{\tau_{b,Y}^- < T_{a,Y}^{+,\lambda} \}} W^{(q+\gamma_p)}(Y_{\tau_{b,Y}^-} - u) \Bigr) \dd u \notag \\
				&= \frac{1}{Z^{(q+\gamma_p)}(b)} \Bigl( g^{(q,p)}_{\lambda,\gamma}(b)+\gamma \int^b_0 W^{(q+\gamma_p)}(b-u) \dd u \Bigr) \,  \Bigl( \overline{z}_b^{(q+\gamma_p,-\gamma_p)}(x) \notag \\
				& \quad - \frac{\W^{(q)}(x-b)}{\Z^{(q)}(a-b,\varphi_{q+\lambda})} \lambda \int^\infty_0 \ee^{-\varphi_{q+\lambda}y}\,\overline{z}_b^{(q+\gamma_p,-\gamma_p)}(y+a) \dd y \Bigr) \notag \\
				&\quad - \gamma \int^b_0 \Bigl( \overline{w}_b^{(q+\gamma_p,-\gamma_p)}(x;u) - \frac{\W^{(q)}(x-b)}{\Z^{(q)}(a-b,\varphi_{q+\lambda})} \lambda \int^\infty_0 \ee^{-\varphi_{q+\lambda}y}\,\overline{w}_b^{(q+\gamma_p,-\gamma_p)}(y+a;u) \dd y \Bigr) \dd u,
				\label{eq:OccupationProp2-NewIntermediate}
			\end{align}
			where the second last equality follows by a substitution of Eq.~\eqref{eq:OccupationProp2-SMPx<bIdentity1} and Fubini's theorem, and the last equality follows by using Lemma \ref{lem:SimilartoRonnieLemma}.
			
			Now,  letting  $x=b$ in the above and solving w.r.t.~$g^{(q,p)}_{\lambda,\gamma}(b)$, we find that
			\begin{equation}
				g^{(q,p)}_{\lambda,\gamma}(b) =  \frac{\int^\infty_0 \ee^{-\varphi_{q+\lambda}y} \, \Bigl( \gamma \int^b_0 \overline{w}_b^{(q+\gamma_p,-\gamma_p)}(y+a;u) \dd u\Bigr) \dd y}{\int^\infty_0 \ee^{-\varphi_{q+\lambda}y} \,  \overline{z}_b^{(q+\gamma_p,-\gamma_p)}(y+a) \dd y} Z^{(q+\gamma_p)}(b) - \gamma \int^b_0 W^{(q+\gamma_p)}(b-u) \dd u. \notag
			\end{equation}
			Lastly, by noticing that $\overline{w}_b^{(q+\gamma_p,-\gamma_p)}(x;u) = W^{(q+\gamma_p)}(x-u)$ and $\overline{z}_b^{(q+\gamma_p,-\gamma_p)}(x) = Z^{(q+\gamma_p)}(x)$ for $x,u \in [0,b)$, we substitute the above quantity into  Eqs.~\eqref{eq:OccupationProp2-SMPx<bIdentity1} and \eqref{eq:OccupationProp2-NewIntermediate} to obtain the form of Eq.~\eqref{Eq:OccupationProp2-Identity1} for $x \in [0,a]$ which concludes the proof.
		\end{proof}
		\begin{remark}\label{Cor:OccupationTimes2}\upshape{
			By noticing that 
			\begin{equation}
				\lim\limits_{a \rightarrow \infty} \frac{\overline{w}_b^{(q+\gamma_p,-\gamma_p)}(y+a;u)}{\W^{(q)}(a)} = \ee^{\varphi_q (y+b)} \alpha_1^{(q,q+\gamma_p)}(b;u), \quad \lim\limits_{a \rightarrow \infty} \frac{\overline{z}_b^{(q+\gamma_p,-\gamma_p)}(y+a;u)}{\W^{(q)}(a)} = \ee^{\varphi_q (y+b)} \alpha_2^{(q,q+\gamma_p)}(b), \notag
			\end{equation}
			and that 
			\begin{align}
				\lim\limits_{a \rightarrow \infty} \frac{\int^\infty_0 \ee^{-\varphi_{q+\lambda}y}\, \Bigl( \gamma \int^b_0 \overline{w}_b^{(q+\gamma_p,-\gamma_p)}(y+a;u) \dd u\Bigr) \dd y}{\int^\infty_0 \ee^{-\varphi_{q+\lambda}y} \,  \overline{z}_b^{(q+\gamma_p,-\gamma_p)}(y+a) \dd y} &= \frac{\int^\infty_0 \ee^{-\varphi_{q+\lambda}y}\, \Bigl( \gamma \int^b_0 \ee^{\varphi_q (y+b)} \alpha_1^{(q,q+\gamma_p)}(b;u) \dd u\Bigr) \dd y}{\int^\infty_0 \ee^{-\varphi_{q+\lambda}y} \,  \ee^{\varphi_q (y+b)} \alpha_2^{(q,q+\gamma_p)}(b) \dd y} \notag \\
				&= \gamma \int^b_0 \frac{  \alpha_1^{(q,q+\gamma_p)}(b;u) }{ \alpha_2^{(q,q+\gamma_p)}(b)} \dd u, \notag
			\end{align}
			we may take $a \rightarrow \infty$ in Eq.~\eqref{Eq:OccupationProp2-Identity1} to obtain 
			\begin{align}
				\E_x \Bigl( \ee^{-q T_{b,U}^{-,\gamma} - p \int_0^{T_{b,U}^{-,\gamma}} 1_{\left\{U_s<b\right\}} \mathrm{d} s} \; \1_{\{T_{b,U}^{-,\gamma} < \infty \}}\Bigr) &=   \overline{z}_b^{(q+\gamma_p,-\gamma_p)}(x) \cdot \gamma \int^b_0 \frac{  \alpha_1^{(q,q+\gamma_p)}(b;u) }{ \alpha_2^{(q,q+\gamma_p)}(b)} \dd u  - \gamma \int^b_0 \overline{w}_b^{(q+\gamma_p,-\gamma_p)}(x;u) \dd u. \notag  
			\end{align} 
				}
			\end{remark}
			\noindent Combining the above results of Propositions \ref{Prop:OccupationTimes1} and \ref{Prop:OccupationTimes2}, we have the following two-sided occupation times.
			\begin{proposition}\label{Prop:JointOccupationTimes}
				Denote $T^{\lambda,\gamma} := T_{a,U}^{+,\lambda} \wedge T_{b,U}^{-,\gamma}$. Let $p,q \geq 0$ and $\lambda, \gamma> 0$. Then, for $x,b \in [0,a]$, we have that
				\begin{align}
					\E_x \Bigl(  \ee^{-q T^{\lambda,\gamma} - p \int_0^{T^{\lambda,\gamma}} 1_{\left\{U_s<b\right\}} \mathrm{d} s} \; \1_{\{T^{\lambda,\gamma} < \infty\}}\Bigr) &= \frac{\int^\infty_0 \ee^{-\varphi_{q+\lambda}y}\, \Bigl(1 +  \gamma \int^b_0 \overline{w}_b^{(q+\gamma_p,-\gamma_p)}(y+a;u) \dd u\Bigr) \dd y}{\int^\infty_0 \ee^{-\varphi_{q+\lambda}y} \,  \overline{z}_b^{(q+\gamma_p,-\gamma_p)}(y+a) \dd y} \overline{z}_b^{(q+\gamma_p,-\gamma_p)}(x) \notag \\
					&\quad - \gamma \int^b_0 \overline{w}_b^{(q+\gamma_p,-\gamma_p)}(x;u) \dd u. \label{Eq:JointOccupationProp-Identity1}
				\end{align} 
			\end{proposition}
			
			\begin{proof}
				Noticing that 
				\begin{align}
					\E_x \Bigl( \ee^{-q T^{\lambda,\gamma} - p \int_0^{T^{\lambda,\gamma}} 1_{\left\{U_s<b\right\}} \mathrm{d} s} \; \1_{\{T^{\lambda,\gamma} < \infty\}}\Bigr) &= 	\E_x \Bigl( \ee^{-q T_{a,U}^{+,\lambda} - p \int_0^{T_{a,U}^{+,\lambda}} 1_{\left\{U_s<b\right\}} \mathrm{d} s} \; \1_{\{T_{a,U}^{+,\lambda} < T_{b,U}^{-,\gamma}\}}\Bigr) \notag \\
					&\quad + \E_x \Bigl( \ee^{-q T_{b,U}^{-,\gamma} - p \int_0^{T_{b,U}^{-,\gamma}} 1_{\left\{U_s<b\right\}} \mathrm{d} s} \; \1_{\{T_{b,U}^{-,\gamma} < T_{a,U}^{+,\lambda}\}}\Bigr), \notag
				\end{align}
				and using Eqs.~\eqref{Eq:OccupationProp1-Identity1} and \eqref{Eq:OccupationProp2-Identity1} as well as that $1/\varphi_{q+\lambda} = \int^\infty_0 \ee^{-\varphi_{q+\lambda} y} \dd y$, Eq.~\eqref{Eq:JointOccupationProp-Identity1} immediately follows. 
			\end{proof}
			\noindent Finally, with the purpose of completing the full spectrum of occupation time results, we provide the following occupation times identities without proof, since these proofs are similar to that of  Propositions \ref{Prop:OccupationTimes1} and \ref{Prop:OccupationTimes2}.  
			
			\begin{proposition}
				Let $p,q \geq 0$ and $\lambda,\gamma > 0$. Then, for $x,b \in [0,a]$, we have that
				\begin{align}
					&\mathbb{E}_x \Bigl(\ee^{-q T_{a,U}^{-,\lambda} - p\int^{T_{a,U}^{-,\lambda}}_0 \1_{\{U_s < b\}} \dd s} 1_{\left\{  T_{a,U}^{-,\lambda} < T_{b,U}^{+,\gamma} \right\}} \Bigr) =  \frac{\overline{z}_b^{(q+\lambda_p,\gamma-\lambda_p)}(x)  + \lambda \int^x_b \W^{(q+\gamma+\lambda)}(x-y)  \overline{z}_b^{(q+\lambda_p,\gamma-\lambda_p)}(y) \dd y }{\ee^{\varphi_{q+\gamma} a} \alpha_2^{(q+\gamma,q+\lambda_p)}(b) +  \lambda \int^a_b \Z^{(q+\gamma+\lambda)}(a-y,\varphi_{q+\gamma})\overline{z}_b^{(q+\lambda_p,\gamma-\lambda_p)}(y) \dd y}  \notag \\
					&\quad \times \Bigl( \lambda \int^b_0 \Bigl[ \ee^{\varphi_{q+\gamma} a} \alpha_1^{(q+\gamma,q+\lambda_p)}(b;u) +  \lambda \int^a_b \Z^{(q+\gamma+\lambda)}(a-y,\varphi_{q+\gamma})\overline{w}_b^{(q+\lambda_p,\gamma-\lambda_p)}(y;u) \dd y  \Bigr] \dd u \notag \\
					& \quad + \lambda \int^a_b \Z^{(q+\gamma+\lambda)}(a-u,\varphi_{q+\gamma}) \dd u \Bigr) -  \lambda \int^a_0 \Bigl(
					\overline{w}_b^{(q+\lambda_p,\gamma-\lambda_p)}(x;u) + \lambda \int^x_b \W^{(q+\gamma+\lambda)}(x-y)  \overline{w}_b^{(q+\lambda_p,\gamma-\lambda_p)}(y;u) \dd y \Bigr) \dd u , \notag
				\end{align}
				and 	
				\begin{align}
					\mathbb{E}_x \Bigl(\ee^{-q T_{b,U}^{+,\gamma} - p\int^{T_{b,U}^{+,\gamma}}_0 \1_{\{U_s < b\}} \dd s} 1_{\left\{  T_{b,U}^{+,\gamma} < T_{a,U}^{-,\lambda} \right\}} \Bigr) &=  \frac{\overline{z}_b^{(q+\lambda_p,\gamma-\lambda_p)}(x)  + \lambda \int^x_b \W^{(q+\gamma+\lambda)}(x-y)  \overline{z}_b^{(q+\lambda_p,\gamma-\lambda_p)}(y) \dd y }{\ee^{\varphi_{q+\gamma} a} \alpha_2^{(q+\gamma,q+\lambda_p)}(b) +  \lambda \int^a_b \Z^{(q+\gamma+\lambda)}(a-y,\varphi_{q+\gamma})\overline{z}_b^{(q+\lambda_p,\gamma-\lambda_p)}(y) \dd y}  \notag \\
					&\quad \times \gamma \int^\infty_b  \Z^{(q+\gamma+\lambda)}(a-u,\varphi_{q+\gamma}) - \gamma \int^x_b \W^{(q+\gamma+\lambda)}(x-u)\dd u. \notag
				\end{align}
			\end{proposition}
			It is then clear that the occupation time w.r.t.~$T_{b,U}^{+,\gamma} \wedge T_{a,U}^{-,\lambda}$ follows by using the above two quantities and a similar proof to that of Proposition \ref{Prop:JointOccupationTimes}.

			\section*{Appendix}\label{appendix}

			\subsection*{Proof of Lemma \ref{lem:RefractedConvolutionIdentities}.}

			To prove (i) -- (iv), we first require some auxiliary identities. In particular, for $p, p+q, x \geq 0$ and $x \in [0,a]$, we have the following identity  (see p.~1176 from \cite{F2014}):
			\begin{equation}
				\begin{aligned}
					& \delta \int_0^x \mathbb{W}^{(p+q)}(x-y) W^{(p)}(y) \mathrm{d} y+q \int_0^x \int_0^y \mathbb{W}^{(p+q)}(y-z) W^{(p)}(z) \mathrm{d} z \mathrm{~d} y \newline \\
					&=\int_0^x \mathbb{W}^{(p+q)}(y) \mathrm{d} y-\int_0^x W^{(p)}(y) \mathrm{d} y.
				\end{aligned} \label{eq:RenaudConvolutionIdentity}
			\end{equation}
			Then, by differentiating Eq.~\eqref{eq:RenaudConvolutionIdentity} with respect to $x$, we obtain 
			\begin{equation}
				\begin{aligned}
					& q \int_0^x \mathbb{W}^{(p+q)}(x-y) W^{(p)}(y) \mathrm{d} y =\mathbb{W}^{(p+q)}(x) -W^{(p)}(x) - \delta \int_{[0, x)} \mathbb{W}^{(p+q)}(x-y) W^{(p)}(\mathrm{d} y),
				\end{aligned} \label{eq:WangMainConvolutionIdentity}
			\end{equation}
			and by using the definition of $Z^{(p)}(\cdot)$, Eq.~\eqref{eq:RenaudConvolutionIdentity} can be rearranged to yield 
			\begin{equation}
				q \int_0^x \mathbb{W}^{(p+q)}(x-y) Z^{(q)}(y) \mathrm{d} y= \mathbb{Z}^{(p+q)}(x) - Z^{(p)}(x) - \delta p \int_0^x \mathbb{W}^{(p+q)}(x-y) W^{(p)}(y) \mathrm{d} y. \label{eq:WangZConvolutionIdentity}
			\end{equation}
			(i) Since $w^{(p)}(x-u) = W^{(p)}(x-u)$ for $x,u \in [0,b)$, and since for $p\geq 0$ and $y <0$ we have that $W^{(p)} (y) = 0$, it is easily seen that Eq.~\eqref{eq:RefractedWConvolutionIdentity} holds since it is reduces to that in Eq.~\eqref{eq:WangMainConvolutionIdentity}. 
			
			Now, for the case of $x \geq b$, observe that
			\begin{align}
				&q \int_0^x \mathbb{W}^{(p+q)}(x-y) w^{(p)}(y;u) \mathrm{d} y \notag \\
				&= q \int_0^x \mathbb{W}^{(p+q)}(x-y) W^{(p)}(y-u) \dd y +  q \int_0^x \mathbb{W}^{(p+q)}(x-y) \Bigl(\delta \int^{y-u}_{b-u} \W^{(p)}(y-u-v) W^{(p)\prime}(v) \dd v  \Bigr)\dd y  \notag \\
				&= q \int_{0}^{x-u} \mathbb{W}^{(p+q)}(x-u-y) W^{(p)}(y) \dd y +  \delta \int_{b-u}^{x-u}  \Bigl( q \int^{x-u}_{0} \mathbb{W}^{(p+q)}(x-u-y) \W^{(p)}(y-v) \dd y  \Bigr) W^{(p)\prime}(v) \dd v \notag \\
				&= q \int_{0}^{x-u} \mathbb{W}^{(p+q)}(x-u-y) W^{(p)}(y) \dd y +  \delta \int_{b-u}^{x-u}  \Bigl( \W^{(p+q)}(x-u-v) -  \W^{(p)}(x-u-v) \Bigr) W^{(p)\prime}(v) \dd v \notag \\
				&= \W^{(p+q)}(x-u) - W^{(p)}(x-u) - \delta \int_{[0,x-u)}\W^{(p+q)}(x-u-y) W^{(p)}(\dd y) \notag \\
				& \eqsp +  \delta \int_{b-u}^{x-u}  \W^{(p+q)}(x-u-v) W^{(p)\prime}(v) \dd v 
				- \delta \int_{b-u}^{x-u}   \W^{(p)}(x-u-v) W^{(p)\prime}(v) \dd v \notag \\
				&= 	\W^{(p+q)}(x-u) - w^{(p)}(x;u) - \delta \int_{[0,b-u)} \W^{(p+q)}(x-u-y) W^{(p)} (\dd y), \label{eq:Lemma-IdentityforConclusion}
			\end{align}
			for which the second equality uses that $W^{(p)}(y) = 0$ for $y<0$, the third equality follows by using Eq.~\eqref{eq:LoeffenLTIdentity1} and the fourth follows by using Eq.~\eqref{eq:WangMainConvolutionIdentity}. 
			
			\medskip
			
			\noindent (ii) Since $z^{(p)}(x) = Z^{(p)}(x)$ for $x \in [0,b)$, and since for $p\geq 0$ and $y <0$ we have that $W^{(p)} (y) = 0$, it easily seen that Eq.~\eqref{eq:RefractedZConvolutionIdentity} holds since it is reduces to that in Eq.~\eqref{eq:WangZConvolutionIdentity}.  For the case of $x \geq b$, a similar argument as that used to derive Eq.~\eqref{eq:Lemma-IdentityforConclusion} by substituting Eqs.~\eqref{eq:LoeffenLTIdentity2} and \eqref{eq:WangZConvolutionIdentity} can be done to complete the proof.
			
			\medskip
			
			\noindent (iii) By using  Eq.~\eqref{eq:RefractedWScaleFunc},
			\begin{align}
				\ovl{w}_b^{(p,q)}(x;u) 
				&= \Bigl( w^{(p)}(x;u) +q \int_b^x \mathbb{W}^{(p+q)}(x-y) w^{(p)}(y;u) \mathrm{d} y \Bigr) \1_{\{ u \in [0,b)\}} \notag \\
				& \eqsp + \Bigl( \W^{(p)}(x-u) +q \int_b^x \mathbb{W}^{(p+q)}(x-y) \W^{(p)}(y-u) \mathrm{d} y \Bigr) \1_{\{u \in [b,\infty)\}} \notag \\
				&= \Bigl( \W^{(p+q)}(x-u) -q \int_0^b \mathbb{W}^{(p+q)}(x-y) w^{(p)}(y;u) \mathrm{d} y  \notag \\
				& \eqsp - \delta \int_{[0,b-u)} \W^{(p+q)}(x-u-y) W^{(p)} (\dd y) \Bigr) \1_{\{ u \in [0,b)\}} \notag \\
				& \eqsp + \Bigl( \W^{(p+q)}(x-u) -q \int_{-u}^{b-u} \mathbb{W}^{(p+q)}(x-u-y) \W^{(p)}(y) \mathrm{d} y \Bigr) \1_{\{u \in [b,\infty)\}} \notag \\
				&= \W^{(p+q)}(x-u) -q \int_0^b \mathbb{W}^{(p+q)}(x-y) W^{(p)}(y-u) \mathrm{d} y \1_{\{ u \in [0,b)\}}  \notag \\
				& \eqsp - \delta \int_{[0,b-u)} \W^{(p+q)}(x-u-y) W^{(p)} (\dd y) \notag \\
				& = \W^{(p+q)}(x-u) -q \int_0^{b-u} \mathbb{W}^{(p+q)}(x-u-y) W^{(p)}(y) \mathrm{d} y \notag \\
				&\eqsp - \delta \int_{[0,b-u)} \W^{(p+q)}(x-u-y) W^{(p)} (\dd y), \notag
			\end{align}
			where the second equality uses  Eq.~\eqref{eq:RefractedWConvolutionIdentity}, and the second last equality follows by using that $w^{(p)}(y;u) = W^{(p)}(y-u)$ for $u \in [0,b)$. 
			
			\medskip
			
			\noindent 
			(iv) The identity in Eq.~\eqref{eq:SecondGenFunc-RefractedZTilde2} can be proved by using a similar method to identity (iii) and Eq.~\eqref{eq:RefractedZConvolutionIdentity}.
			
			\medskip
			
			\noindent 
			(v) The identities follow by using Lemma 1 in \cite{F2014}.
			\subsection*{Proof of Lemma \ref{lem:PoissonPotentialsAndFluctuations}.}
			\textit{\upshape{(i)}} Observe by Eq.~(3.3) and Theorem 3.1, Eq.~(3.18) in \cite{LLWX2018} that, for $x \in [0,a]$ and $y \geq 0$, 
			$$
			\E_x\Bigl( \int^\infty_0 e^{-q t} \1_{\{Y_t \in \dd y, \; t < T_{a,Y}^{+,\lambda} \wedge \tau_{0,Y}^- \}} \dd t \Bigr) = \Bigl( \frac{\W^{(q)}(x)}{\Z^{(q)}(a, \varphi_{q+\lambda})} \Z^{(q)}(a-y, \varphi_{q+\lambda}) - \W^{(q)}(x-y)\Bigr) \dd y.
			$$
			Then, by using spatial homogeneity, the first identity follows. 
			
			For the second identity, observe from Eq.~(3.3) and Corollary 3.1 in \cite{LLWX2018} that, for $x \in [0,a]$ and $y \geq a$, 
			$$
			\E_x\bigl( \ee^{-q T_{a,Y}^{+,\lambda}} \1_{\{Y_{T_{a,Y}^{+,\lambda}} \in \dd y, \; T_{a,Y}^{+,\lambda} < \tau_{0,Y}^-\}} \bigr) = \lambda \, \Bigl( \frac{\W^{(q)}(x)}{\Z^{(q)}(a, \varphi_{q+\lambda})} \Z^{(q)}(a-y, \varphi_{q+\lambda}) - \W^{(q)}(x-y)\Bigr) \, \dd y. 
			$$
			Now, the second identity follows since by using spatial homogeneity, that $\Z^{(q)}(a-y, \varphi_{q+\lambda}) = \ee^{\varphi_{q+\lambda}(a-y)}$ and that $\W^{(q)}(x-y) = 0$ for $x \in [b,a]$ and $y > a$.
			
			\medskip
			
			\noindent \textit{\upshape{(ii)}} The proof follows similarly as for that in part \upshape{(i)}. Indeed, the first quantity follows from Eq.~(3.10) and Theorem 3.1, Eq~(3.14) in \cite{LLWX2018}, whilst the second quantity follows from Corollary 3.1 in \cite{LLWX2018}.
			\color{black}
			
			\subsection*{Proof of Lemma \ref{lem:SimilartoRonnieLemma}.}
			(i) For  $\jj = 1,2$,  using the strong Markov property, we observe that 
			\begin{align}
				\E_x\bigl( \ee^{-q\tau_{b,Y}^-} \1_{\{\tau_{b,Y}^- < T_{a,Y}^{+,\lambda}\}} f^{(p)}_{1,\jj}(Y_{\tau_{b,Y}^- }) \bigr)
				=& \;  \E_x\bigl( \ee^{-q\tau_{b,Y}^-} \1_{\{\tau_{b,Y}^- < \infty\}} f^{(p)}_{1,\jj}(Y_{\tau_{b,Y}^- }) \bigr)  \notag \\
				&\quad  - \E_x\bigl( \ee^{-q T_{a,Y}^{+,\lambda}} \1_{\{T_{a,Y}^{+,\lambda} < \tau_{b,Y}^-\}} \E_{Y_{T_{a,Y}^{+,\lambda}}}\bigl( \ee^{-q \tau_{b,Y}^-} \1_{\{\tau_{b,Y}^- < \infty \}} f^{(p)}_{1,\jj}(Y_{\tau_{b,Y}^- }) \bigr) \bigr). \label{lem2-eq:MainSMPIdentity}
			\end{align}
			In order to evaluate the expectations in the above equation, we recall Eqs.~\eqref{eq:W(q)JumpDownwardsinYDynamics}--\eqref{eq:Z(q)JumpDownwardsinYDynamics}, from which it yields that 
			\begin{align}
				\E_x\bigl( \ee^{-q\tau_{b,Y}^-} \1_{\{\tau_{b,Y}^- < \infty \}} f^{(p)}_{1,\jj}(Y_{\tau_{b,Y}^- }) \bigr) 
				&= \lim\limits_{a \rightarrow \infty} \E_x\bigl( \ee^{-q\tau_{b,Y}^-} \1_{\{\tau_{b,Y}^- < \tau_{a,Y}^+\}} f^{(p)}_{1,\jj}(Y_{\tau_{b,Y}^-}) \bigr), \notag \\
				&= f^{(p,q)}_{2,\jj}(x) - \W^{(q)}(x-b) \times \lim\limits_{a \rightarrow \infty} \frac{f^{(p,q)}_{2,\jj}(a)}{\W^{(q)}(a-b)}, \label{eq:NecessaryFormforLemma}
			\end{align}
			for which it is easy to verify, from Eqs.~\eqref{eq:SecondGenFunc-RefractedWTilde2} and \eqref{eq:SecondGenFunc-RefractedZTilde2} along with Eq.~\eqref{eq:LimitofRatioofScaleFuncs} and dominated convergence, that 
			\begin{equation}
				\lim\limits_{a\uparrow \infty} \frac{f^{(p,q)}_{2,1}(a)}{\W^{(q)}(a-b)} = \ee^{\varphi_q b} \alpha_1^{(q,p)}(b;u), \quad \text{ and } \quad  \lim\limits_{a\uparrow \infty} \frac{f^{(p,q)}_{2,2}(a)}{\W^{(q)}(a-b)} = \ee^{\varphi_q b} \alpha_2^{(q,p)}(b). \label{eq:LimitOfWandZBar}
			\end{equation} 
			Hence, denoting $c_1^{(q,p)} := \ee^{\varphi_q b} \alpha_1^{(q,p)}(b;u)$ and $c_2^{(q,p)} := \ee^{\varphi_q b} \alpha_2^{(q,p)}(b)$,
			\begin{equation}
				\E_x\bigl( \ee^{-q\tau_{b,Y}^-} \1_{\{\tau_{b,Y}^- < \infty\}} f^{(p)}_{1,\jj}(Y_{\tau_{b,Y}^- }) \bigr) =  f^{(p,q)}_{2,\jj}(x) - c_\jj^{(q,p)} \W^{(q)}(x-b),  \label{lem-eq:InitialLimitQuantity}
			\end{equation}
			and thus Eq.~\eqref{lem2-eq:MainSMPIdentity} becomes
			\begin{align}
				\E_x\bigl( \ee^{-q\tau_{b,Y}^-} \1_{\{\tau_{b,Y}^- < T_{a,Y}^{+,\lambda}\}} f^{(p)}_{1,\jj}(Y_{\tau_{b,Y}^- }) \bigr)
				&=  f^{(p,q)}_{2,\jj}(x) - c_\jj^{(q,p)} \W^{(q)}(x-b) - \E_x\bigl( \ee^{-q T_{a,Y}^{+,\lambda}} \1_{\{T_{a,Y}^{+,\lambda} < \tau_{b,Y}^-\}} f^{(p)}_{2,\jj}(Y_{T_{a,Y}^{+,\lambda}}) \bigr) \notag \\
				&\eqsp + c_\jj^{(q,p)} \E_x\bigl( \ee^{-q T_{a,Y}^{+,\lambda}} \1_{\{T_{a,Y}^{+,\lambda} < \tau_{b,Y}^-\}} \W^{(q)}(Y_{T_{a,Y}^{+,\lambda}}-b) \bigr). \label{eq:AuxiliaryEqforSecondLemma} \\
				&= f^{(p,q)}_{2,\jj}(x) - c_\jj^{(q,p)} \W^{(q)}(x-b) - \int^\infty_a \E_x\bigl( \ee^{-q T_{a,Y}^{+,\lambda}} \1_{\{Y_{T_{a,Y}^{+,\lambda}} \in \dd y, \;T_{a,Y}^{+,\lambda} < \tau_{b,Y}^-\}} \bigr)  f^{(p,q)}_{2,\jj}(y) \notag \\
				&\quad  + c_\jj^{(q,p)} \int^\infty_a \E_x\bigl( \ee^{-q T_{a,Y}^{+,\lambda}} \1_{\{Y_{T_{a,Y}^{+,\lambda}} \in \dd y, \;T_{a,Y}^{+,\lambda} < \tau_{b,Y}^-\}}\bigr) \W^{(q)}(y-b) . 
				\label{lem2-eq:SMPIdentity2}
			\end{align} 
			Finally,  using Lemma \ref{lem:PoissonPotentialsAndFluctuations} (ii) and Eq.~\eqref{eq:GeneralisedZFunc}, we obtain that
			\begin{align}
				\int^\infty_a \E_x\bigl( \ee^{-q T_{a,Y}^{+,\lambda}} \1_{\{Y_{T_{a,Y}^{+,\lambda}} \in \dd y, \;T_{a,Y}^{+,\lambda} < \tau_{b,Y}^-\}} \bigr) \W^{(q)}(y-b) =& \; \frac{\W^{(q)}(x-b)}{\Z^{(q)}(a-b,\varphi_{q+\lambda})} \lambda \int^\infty_a  \ee^{\varphi_{q+\lambda}(a-y)} \W^{(q)}(y-b)\dd y \notag \\
				=& \; \W^{(q)}(x-b), \notag
			\end{align}
			and
			\begin{equation}
				\int^\infty_a \E_x\bigl( \ee^{-q T_{a,Y}^{+,\lambda}} \1_{\{Y_{T_{a,Y}^{+,\lambda}} \in \dd y, \;T_{a,Y}^{+,\lambda} < \tau_{b,Y}^-\}} \bigr) f^{(p,q)}_{2,\jj}(y) =  \frac{\W^{(q)}(x-b)}{\Z^{(q)}(a-b,\varphi_{q+\lambda})} \lambda \int^\infty_0  \ee^{-\varphi_{q+\lambda} y} f^{(p,q)}_{2,\jj}(y+a) \dd y. \notag
			\end{equation}
			The proof is completed by substituting the above two quantities into Eq.~\eqref{lem2-eq:SMPIdentity2}. 
			
			\medspace
			
			\noindent (ii) By following similar steps as in the proof of (i), we obtain the identities in Eqs.~\eqref{lem2-eq:MainSMPIdentity} and \eqref{eq:AuxiliaryEqforSecondLemma} with $T_{a,Y}^{+,\lambda}$ replaced by $T_{a,Y}^{-,\lambda}$. Hence,
			%
			\begin{align}
				\E_x\bigl( \ee^{-q\tau_{b,Y}^-} \1_{\{\tau_{b,Y}^- < T_{a,Y}^{-,\lambda}\}} f^{(p)}_{1,\jj}(Y_{\tau_{b,Y}^- }) \bigr)
				&=  f^{(p,q)}_{2,\jj}(x) - c_\jj^{(q,p)} \W^{(q)}(x-b) - \E_x\bigl( \ee^{-q T_{a,Y}^{-,\lambda}} \1_{\{T_{a,Y}^{-,\lambda} < \tau_{b,Y}^-\}} f^{(p)}_{2,\jj}(Y_{T_{a,Y}^{-,\lambda}}) \bigr) \notag \\
				&\eqsp + c_\jj^{(q,p)} \E_x\bigl( \ee^{-q T_{a,Y}^{-,\lambda}} \1_{\{T_{a,Y}^{-,\lambda} < \tau_{b,Y}^-\}} \W^{(q)}(Y_{T_{a,Y}^{-,\lambda}}-b) \bigr) \notag \\
				&= f^{(p,q)}_{2,\jj}(x) - c_\jj^{(q,p)} \W^{(q)}(x-b) - \int^a_b \E_x\bigl( \ee^{-q T_{a,Y}^{-,\lambda}} \1_{\{Y_{T_{a,Y}^{-,\lambda}} \in \dd y, \;T_{a,Y}^{-,\lambda} < \tau_{b,Y}^-\}} \bigr)  f^{(p,q)}_{2,\jj}(y) \notag \\
				&\quad  + c_\jj^{(q,p)} \int^\infty_a \E_x\bigl( \ee^{-q T_{a,Y}^{-,\lambda}} \1_{\{Y_{T_{a,Y}^{-,\lambda}} \in \dd y, \;T_{a,Y}^{-,\lambda} < \tau_{b,Y}^-\}}\bigr) \W^{(q)}(y-b) \notag \\
				&= f^{(p,q)}_{2,\jj}(x) + \lambda \int^x_b \W^{(q+\lambda)}(x-y)f^{(p,q)}_{2,\jj}(y)\dd y \notag \\ 
				&\quad-\frac{\W^{(q+\lambda)}(x-b)}{\Z^{(q+\lambda)}(a-b,\varphi_{q})} \lambda \int^a_b \Z^{(q+\lambda)}(a-y,\varphi_q) f_{2,\jj}^{(p,q)}(y) \dd y \notag \\
				&\quad - c_\jj^{(q,p)} \Bigl( \W^{(q)}(x-b) + \lambda \int^a_b \W^{(q+\lambda)}(x-y)\W^{(q)}(y-b) \dd y \notag \\
				&\quad\quad  - \frac{\W^{(q+\lambda)}(x-b)}{\Z^{(q+\lambda)}(a-b,\varphi_{q})} \lambda \int^a_b \Z^{(q+\lambda)}(a-y,\varphi_q) \W^{(q)}(y-b) \dd y \Bigr) \notag \\
				&= f^{(p,q)}_{2,\jj}(x) + \lambda \int^x_b \W^{(q+\lambda)}(x-y)f^{(p,q)}_{2,\jj}(y)\dd y \notag \\
				&\quad - \frac{\W^{(q+\lambda)}(x-b)}{\Z^{(q+\lambda)}(a-b,\varphi_{q})}\Bigl( c_\jj^{(q,p)} e^{\varphi_q (a-b)} + \lambda \int^a_b \Z^{(q+\lambda)}(a-y,\varphi_q) f_{2,\jj}^{(p,q)}(y) \dd y \Bigr), \notag
			\end{align} 
			where the second last equality follows by using Lemma \ref{lem:PoissonPotentialsAndFluctuations} (ii) Eq.~\eqref{eq:Landriault-Identity4}, and the last equality follows by using Eqs.~\eqref{eq:GeneralisedZFunc}, \eqref{eq:LoeffenLTIdentity1} and \eqref{eq:LoeffenLTIdentity2}.
			
			\section*{Acknowledgement}
			The authors are grateful to the anonymous referees for their constructive comments and suggestions that have improved the content and presentation of this paper.

	\bibliography{ApostolosV2.bib}
	\bibliographystyle{abbrv}
\end{document}